\documentclass[12pt]{amsart}

\usepackage{mathpazo}

\theoremstyle{plain}
\newtheorem{theorem}{Theorem}[section]
\newtheorem{proposition}[theorem]{Proposition}
\newtheorem{lemma}[theorem]{Lemma}
\newtheorem{claim}[theorem]{Claim}
\newtheorem{fact}[theorem]{Fact}
\newtheorem{question}[theorem]{Question}
\newtheorem{conjecture}[theorem]{Conjecture}

\theoremstyle{definition}
\newtheorem{definition}[theorem]{Definition}

\theoremstyle{remark}
\newtheorem{remark}[theorem]{Remark}

\DeclareMathAlphabet{\mathpzc}{OT1}{pzc}{m}{it}
\newcommand{\Ts}{\mathpzc{Ts}}
\newcommand{\Tss}{\mathpzc{Ts}^{*}}
\newcommand{\G}{\mathbb{G}}
\newcommand{\sel}{\mathfrak{S}}
\newcommand{\T}{\mathcal{T}}
\newcommand{\m}{\mathfrak{m}}

\begin{document}

\title[Tsirelson-like spaces and complexity]{Tsirelson-like spaces and complexity of classes of Banach spaces}
\author{Ond\v{r}ej Kurka}
\thanks{The research was supported by the grant GA\v{C}R 14-04892P. The author is a junior researcher in the University Centre for Mathematical Modelling, Applied Analysis and Computational Mathematics (MathMAC)}
\address{Department of Mathematical Analysis, Charles University, Soko\-lovsk\'a 83, 186 75 Prague 8, Czech Republic}
\email{kurka.ondrej@seznam.cz}
\begin{abstract}
Employing a construction of Tsirelson-like spaces due to Argyros and Deliyanni, we show that the class of all Banach spaces which are isomorphic to a subspace of $ c_{0} $ is a complete analytic set with respect to the Effros Borel structure of separable Banach spaces. Moreover, the classes of all separable spaces with the Schur property and of all separable spaces with the Dunford-Pettis property are $ \mathbf{\Pi}^{1}_{2} $-complete.
\end{abstract}
\maketitle

\section{Introduction and main results}

During the last two decades, it turned out that descriptive set theory provides a fruitful approach to several questions in separable Banach space theory. A particular and generally still not well understood question is the question of the descriptive complexity of a given class of separable Banach spaces. In the present work, we introduce a new approach to complexity problems in Banach space theory which is based on a fundamental example of Tsirelson.

The connections between descriptive set theory and Banach space theory were discovered by J.~Bourgain \cite{bourgain1, bourgain2}. Later, B.~Bossard \cite{bossard} investigated codings of separable Banach spaces up to isomorphism by standard Borel spaces and used the Effros Borel structure for studying complexity questions in Banach space theory (see Section~\ref{sec:prelimI} for the definitions of the Effros Borel structure and of the related notions used below, let us note here that by an isomorphism we mean a linear isomorphism throughout this paper).

It can be shown quite easily that the isomorphism class of any separable Banach space is analytic. B.~Bossard asked in \cite{bossard} whether $ \ell_{2} $ is (up to isomorphism) the only infinite-dimensional separable Banach space whose isomorphism class is Borel. There are several examples for which the isomorphism class is shown to be non-Borel, for instance Pe\l czy\'nski's universal space \cite{bossard}, $ C(2^{\mathbb{N}}) $ (see e.g. \cite[(33.26)]{kechris}) or $ L_{p}([0, 1]) $ for $ 1 < p < \infty, p \neq 2, $ (see e.g. \cite{ghawadrah}). A by-product of the present work are two new examples $ (\bigoplus G_{n})_{c_{0}} $ and $ (\bigoplus G_{n})_{\ell_{1}} $ (see Remarks~\ref{remTss}(ii) and \ref{remTs}(vii)).

Bossard's question has been recently answered by G.~Godefroy \cite{godefroycompl} who has proven the existence of a space which is not isomorphic to $ \ell_{2} $ but the isomorphism class of which is Borel. The following question posed in \cite{godefroyprobl}, however, remains open.

\begin{question}[Godefroy] \label{questgodef}
Is the class of all Banach spaces isomorphic to $ c_{0} $ Borel?
\end{question}

In Section~\ref{sec:question}, we present some remarks concerning this interesting problem. Although we have not found its solution, we have obtained the following related result.

\begin{theorem} \label{thmc0subsp}
The class of all Banach spaces which can be embedded isomorphically into $ c_{0} $ is complete analytic. In particular, it is not Borel.
\end{theorem}

This result answers \cite[Problem~4]{godefroyprobl} and provides most likely the first example of a space $ X $ for which the class of spaces embeddable into $ X $ is shown not to be Borel. In other words, we have proven that the embeddability relation $ Y~\hookrightarrow~X $ has a non-Borel horizontal section $ \cdot~\hookrightarrow~X $. This discovery is not surprising, as the vertical section $ Y~\hookrightarrow~\cdot $ is known to be non-Borel for every infinite-dimensional $ Y $ (see \cite[Corollary~3.3(vi)]{bossard}).

Our second main result is based on a combination of methods used for proving Theorem~\ref{thmc0subsp} with a tree space method used in \cite{kurka}.

\begin{theorem} \label{thmschur}
The classes of all separable Banach spaces with the Schur property and of all separable Banach spaces with the Dunford-Pettis property are $ \mathbf{\Pi}^{1}_{2} $-complete. In particular, these classes are not $ \mathbf{\Sigma}^{1}_{2} $.
\end{theorem}

This result answers two questions posed by B.~M.~Braga in \cite{braga}. We recall that a Banach space $ X $ is said to have the \emph{Schur property} if every weakly convergent sequence in $ X $ is norm convergent. The \emph{Dunford-Pettis property} is defined in Section~\ref{sec:prelimII}. We note here just that a remarkable characterization states that $ X $ has the Dunford-Pettis property if and only if $ x^{*}_{n}(x_{n}) \to x^{*}(x) $ whenever $ x^{*}_{n} \to x^{*} $ weakly in $ X^{*} $ and $ x_{n} \to x $ weakly in $ X $.

Both results above are significantly based on a construction due to S.~A.~Argyros and I.~Deliyanni \cite{argdel} who generalized the well-known example of B.~S.~Tsirelson \cite{tsirelson}. Let us recall the definition of this important example.

For $ E \subset \mathbb{N} $ and $ x \in c_{00}(\mathbb{N}) $, we denote by $ Ex $ the restriction of $ x $ on $ E $, i.e., the element of $  c_{00}(\mathbb{N}) $ given by $ Ex(i) = x(i) $ for $ i \in E $ and $ Ex(i) = 0 $ for $ i \notin E $. A family $ \{ E_{1}, \dots, E_{n} \} $ of successive finite subsets of $ \mathbb{N} $ is said to be \emph{admissible} if
$$ n < E_{1} < E_{2} < \dots < E_{n}. $$
The system of all admissible families is denoted by $ \mathrm{adm} $.

\begin{definition}[Tsirelson]
Let $ \Theta $ be the smallest absolutely convex subset of $ c_{00}(\mathbb{N}) $ containing every basic vector $ e_{i} = \mathbf{1}_{\{ i \}}, i \in \mathbb{N}, $ and satisfying
$$ \{ E_{1}, \dots, E_{n} \} \in \mathrm{adm} \; \& \; x_{1}, \dots, x_{n} \in \Theta \quad \Rightarrow \quad \frac{1}{2} \sum_{k=1}^{n} E_{k}x_{k} \in \Theta. $$
Let $ \Vert \cdot \Vert_{\Tss} $ be the Minkowski gauge of $ \Theta $ and let $ \Tss $ be a completion of $ (c_{00}(\mathbb{N}), \Vert \cdot \Vert_{\Tss}) $.
\end{definition}

The space $ \Tss $ is the first example of an infinite-dimensional Banach space not containing an isomorphic copy of $ c_{0} $ or any $ \ell_{p} $. It is well-known that $ \Tss $ is reflexive and dual to the space $ \Ts $ defined as the Banach space of sequences $ x = \{ x(i) \}_{i=1}^{\infty} $ with the basis $ e_{i} = \mathbf{1}_{\{ i \}} $ and with the implicitly defined norm
$$ \Vert x \Vert_{\Ts} = \max \Bigg\{ \Vert x \Vert_{\infty}, \frac{1}{2} \sup \bigg\{ \sum_{k=1}^{n} \Vert E_{k}x \Vert_{\Ts} : \{ E_{1}, \dots, E_{n} \} \in \mathrm{adm} \bigg\} \Bigg\}. $$

\section{Preliminaries I} \label{sec:prelimI}

Our terminology concerning Banach space theory and descriptive set theory follows \cite{fhhmpz} and \cite{kechris}.

A \emph{Polish space (topology)} means a separable completely metrizable space (topology). A set $ X $ equipped with a $ \sigma $-algebra is called a \emph{standard Borel space} if the $ \sigma $-algebra is generated by a Polish topology on $ X $.

A subset $ A $ of a standard Borel space $ X $ is called an \emph{analytic set} (or a \emph{$ \mathbf{\Sigma}^{1}_{1} $ set}) if there exist a standard Borel space $ Y $ and a Borel subset $ B $ of $ X \times Y $ such that $ A $ is the projection of $ B $ on the first coordinate. The complement of an analytic set is called a \emph{coanalytic set} (or a \emph{$ \mathbf{\Pi}^{1}_{1} $ set}).

A subset $ A $ of a standard Borel space $ X $ is called a \emph{$ \mathbf{\Sigma}^{1}_{2} $ set} if there exist a standard Borel space $ Y $ and a coanalytic subset $ B $ of $ X \times Y $ such that $ A $ is the projection of $ B $ on the first coordinate. The complement of a $ \mathbf{\Sigma}^{1}_{2} $ set is called a \emph{$ \mathbf{\Pi}^{1}_{2} $ set}.

Let $ \Gamma $ be a class of sets in standard Borel spaces (for example $ \mathbf{\Sigma}^{1}_{1} $ or $ \mathbf{\Pi}^{1}_{2} $). A subset $ A $ of a standard Borel space $ X $ is called a \emph{$ \Gamma $-hard set} if every $ \Gamma $ subset $ B $ of a standard Borel space $ Y $ admits a Borel mapping $ f : Y \to X $ such that $ f^{-1}(A) = B $. A subset $ A $ of a standard Borel space $ X $ is called a \emph{$ \Gamma $-complete set} if it is $ \Gamma $ and $ \Gamma $-hard at the same time.

A $ \mathbf{\Sigma}^{1}_{1} $-hard ($ \mathbf{\Sigma}^{1}_{1} $-complete, $ \mathbf{\Pi}^{1}_{1} $-hard, $ \mathbf{\Pi}^{1}_{1} $-complete) set may be called also \emph{hard analytic (complete analytic, hard coanalytic, complete coanalytic)}.

We note that the introduced notion of a hard (complete) set is suitable for classes like $ \mathbf{\Sigma}^{1}_{1} $ or $ \mathbf{\Pi}^{1}_{2} $ but not for Borel classes in Polish spaces. In that case, only a zero-dimensional $ Y $ and a continuous $ f $ are considered.

Let us recall a standard simple argument for $ \Gamma $-hardness of a set.

\begin{lemma} \label{lemmhardset}
Let $ A \subset X $ and $ C \subset Z $ be subsets of standard Borel spaces $ X $ and $ Z $. Assume that $ C $ is $ \Gamma $-hard. If there is a Borel mapping $ g : Z \to X $ such that
$$ g(z) \in A \quad \Leftrightarrow \quad z \in C, $$
then $ A $ is $ \Gamma $-hard as well.
\end{lemma}

For a topological space $ X $, we denote by $ \mathcal{F}(X) $ the family of all closed subsets of $ X $ and by $ \mathcal{K}(X) $ the family of all compact subsets of $ X $.

The \emph{hyperspace of compact subsets of $ X $} is defined as $ \mathcal{K}(X) $ equipped with the \emph{Vietoris topology}, i.e., the topology generated by the sets of the form 
$$ \{ K \in \mathcal{K}(X) : K \subset U \}, $$
$$ \{ K \in \mathcal{K}(X) : K \cap U \neq \emptyset \}, $$
where $ U $ varies over open subsets of $ X $. If $ X $ is Polish, then so is $ \mathcal{K}(X) $.

We will need the following classical result (see e.g. \cite[(27.4)]{kechris}).

\begin{theorem}[Hurewicz] \label{thmhur}
If $ X $ is Polish and $ D \subset X $ is $ G_{\delta} $ but not $ F_{\sigma} $, then $ \{ K \in \mathcal{K}(X) : K \cap D \neq \emptyset \} $ is complete analytic.
\end{theorem}

The set $ \mathcal{F}(X) $ of all closed subsets of $ X $ can be equipped with the \emph{Effros Borel structure}, defined as the $ \sigma $-algebra generated by the sets
$$ \{ F \in \mathcal{F}(X) : F \cap U \neq \emptyset \}, $$
where $ U $ varies over open subsets of $ X $. If $ X $ is Polish, then, equipped with this $ \sigma $-algebra, $ \mathcal{F}(X) $ forms a standard Borel space.

It is well-known that the space $ C([0, 1]) $ contains an isometric copy of every separable Banach space. By the \emph{standard Borel space of separable Banach spaces} we mean
$$ \mathcal{SE}(C([0, 1])) = \big\{ F \in \mathcal{F}(C([0, 1])) : \textrm{$ F $ is linear} \big\}, $$
considered as a subspace of $ \mathcal{F}(C([0, 1])) $.

Whenever we say that a class of separable Banach spaces has a property like being analytic, complete analytic, $ \mathbf{\Pi}^{1}_{2} $-complete etc., we consider the class as a subset of $ \mathcal{SE}(C([0, 1])) $.

By $ c_{00}(\Lambda) $ we denote the vector space of all systems $ x = \{ x(\lambda) \}_{\lambda \in \Lambda} $ of scalars such that $ x(\lambda) = 0 $ for all but finitely many $ \lambda $'s. By the canonical basis of $ c_{00}(\Lambda) $ we mean the algebraic basis consisting of vectors $ \mathbf{1}_{\{\lambda\}}, \lambda \in \Lambda $. Instead of $ c_{00}(\mathbb{N}) $, we write simply $ c_{00} $.

In the context of Banach spaces, by a basis we mean a Schauder basis. A basis $ \{ x_{i} \}_{i=1}^{\infty} $ of a Banach space $ X $ is said to be \emph{$ 1 $-unconditional} if $ \Vert \sum_{i \in A} a_{i}x_{i} \Vert \leq \Vert \sum_{i \in B} a_{i}x_{i} \Vert $ whenever $ A \subset B $ are finite sets of natural numbers and $ a_{i} \in \mathbb{R} $ for $ i \in B $.

A basis $ \{ x_{i} \}_{i=1}^{\infty} $ of a Banach space $ X $ is said to be \emph{shrinking} if
$$ X^{*} = \overline{\mathrm{span}} \{ x_{1}^{*}, x_{2}^{*}, \dots \} $$
where $ x_{1}^{*}, x_{2}^{*}, \dots $ is the dual basic sequence $ x_{n}^{*} : \sum_{i=1}^{\infty} a_{i}x_{i} \mapsto a_{n} $. The basis $ \{ x_{i} \}_{i=1}^{\infty} $ is called \emph{boundedly complete} if $ \sum_{i=1}^{\infty} a_{i}x_{i} $ is convergent whenever the sequence of its partial sums is bounded.

Let us recall a classical criterion of reflexivity (see e.g. \cite[Theorem~6.11]{fhhmpz}).

\begin{theorem}[James] \label{thmjames}
Let $ X $ be a Banach space with a basis $ \{ x_{i} \}_{i=1}^{\infty} $. Then $ X $ is reflexive if and only if $ \{ x_{i} \}_{i=1}^{\infty} $ is shrinking and boundedly complete.
\end{theorem}

The remainder of this section is devoted to the proof of the following preliminary result.

\begin{lemma} \label{lemmselect}
Let $ \Xi $ be a standard Borel space and let $ \{ (X_{\xi}, \Vert \cdot \Vert_{\xi}) \}_{\xi \in \Xi} $ be a system of Banach spaces each member of which contains a sequence $ x^{\xi}_{1}, x^{\xi}_{2}, \dots $ whose linear span is dense in $ X_{\xi} $. Assume that the function
$$ \xi \; \mapsto \; \Big\Vert \sum_{k=1}^{n} \lambda_{k} x^{\xi}_{k} \Big\Vert_{\xi} $$
is Borel whenever $ n \in \mathbb{N} $ and $ \lambda_{1}, \dots, \lambda_{n} \in \mathbb{R} $. Then there exists a Borel mapping $ \sel : \Xi \to \mathcal{SE}(C([0, 1])) $ such that $ \sel(\xi) $ is isometric to $ X_{\xi} $ for every $ \xi \in \Xi $.
\end{lemma}

We need to recall some definitions first. Let $ \varepsilon > 0 $ and let $ X, Y $ be Banach spaces. A linear operator $ f : X \to Y $ is called an \emph{$ \varepsilon $-isometry} if
$$ (1 + \varepsilon)^{-1} \Vert x \Vert < \Vert f(x) \Vert < (1 + \varepsilon) \Vert x \Vert, \quad x \in X \setminus \{ 0 \}. $$

A separable Banach space $ \G $ is called \emph{Gurariy} if, for every $ \varepsilon > 0 $, every finite-dimensional Banach spaces $ X $ and $ Y $ with $ X \subset Y $ and every isometry $ f : X \to \G $, there exists some $ \varepsilon $-isometry $ g : Y \to \G $ which extends $ f $. It is known that there exists only one Gurariy space up to isometry (\cite{lusky}, see also \cite{kubsol}).

\begin{lemma}[Kubi\'s, Solecki] \label{lemmkubsol}
Let $ X_{0} $ and $ X_{1} $ be finite-dimensional Banach spaces with $ X_{0} \subset X_{1} $ and let $ f : X_{0} \to \G $ be a $ 2^{-n} $-isometry. Then there is a $ 2^{-(n+1)} $-isometry $ g : X_{1} \to \G $ such that $ \Vert g|_{X_{0}} - f \Vert < 2 \cdot 2^{-n} $.
\end{lemma}

This lemma is proven in \cite{kubsol} and its purpose is to show that $ \G $ contains an isometric copy of every separable Banach space $ X $. Actually, once the lemma is proven, an isometry $ f : X \to \G $ can be found easily. Let $ x_{1}, x_{2}, \dots $ be a dense sequence in $ X $ and let $ X_{n} = \mathrm{span} \, \{ x_{1}, \dots, x_{n} \} $. Then Lemma~\ref{lemmkubsol} allows us to construct a sequence of linear operators $ f_{n} : X_{n} \to \G $ such that $ f_{n} $ is a $ 2^{-n} $-isometry and $ \Vert f_{n+1}|_{X_{n}} - f_{n} \Vert < 2 \cdot 2^{-n} $. For $ x \in X_{m} $, the sequence $ \{ f_{n}(x) \}_{n \geq m} $ is Cauchy. The isometry $ f_{\infty}(x) = \lim_{n \to \infty} f_{n}(x) $ can be extended from $ \bigcup_{m=1}^{\infty} X_{m} $ to an isometry $ f : X \to \G $.

To prove Lemma~\ref{lemmselect}, we use a leftmost branch argument to show that this construction can be accomplished in a Borel measurable way.

\begin{proof}[Proof of Lemma \ref{lemmselect}]
We establish two additional assumptions which make the situation a bit simpler.

(1) We assume that all spaces $ X_{\xi} $ are infinite-dimensional. This is possible, because $ \Xi $ can be decomposed into Borel sets $ \Xi_{d} = \{ \xi \in \Xi : \mathrm{dim} \, X_{\xi} = d \} $ where $ 0 \leq d \leq \infty $. As these sets are Borel, we can deal with every $ \Xi_{d} $ separately. We consider only $ \Xi_{\infty} $ since other $ \Xi_{d} $'s can be handled in a similar way.

(2) We assume moreover that $ x^{\xi}_{n} $ does not belong to the linear span of $ x^{\xi}_{1}, \dots, x^{\xi}_{n-1} $. This is possible, because the sets
$$ \big\{ \xi \in \Xi : x^{\xi}_{n} \in \mathrm{span} \{ x^{\xi}_{1}, \dots, x^{\xi}_{n-1} \} \big\}, \quad n \in \mathbb{N}, $$
are Borel, and so omitting the members which are a linear combination of its predecessors does not disrupt the assumption of the lemma (the resulting sequence will be infinite due to the first additional assumption).

Now, for every $ n \in \mathbb{N} $, let $ u_{n, 1}, u_{n, 2}, \dots $ be a sequence which is dense in the space $ \mathcal{L}(\mathbb{R}^{n}, \G) $ of linear operators from $ \mathbb{R}^{n} $ into $ \G $. Let us define
$$ X^{\xi}_{n} = \mathrm{span} \{ x^{\xi}_{1}, \dots, x^{\xi}_{n} \} $$
and consider operators
$$ u^{\xi}_{n, i} : X^{\xi}_{n} \to \G, \quad u^{\xi}_{n, i} \Big( \sum_{k=1}^{n} \lambda_{k} x^{\xi}_{k} \Big) = u_{n, i} \big( \lambda_{1}, \dots, \lambda_{n} \big). $$
The operators are well defined since we assume that $ x^{\xi}_{1}, \dots, x^{\xi}_{n} $ are linearly independent. Notice that every $ u \in \mathcal{L}(X^{\xi}_{n}, \G) $ can be approximated by some $ u^{\xi}_{n, i} $ with an arbitrarily small error. Therefore, we obtain for every $ \xi \in \Xi $ that
\begin{itemize}
\item there exists $ j \in \mathbb{N} $ such that $ u^{\xi}_{1, j} : X^{\xi}_{1} \to \G $ is a $ 2^{-1} $-isometry,
\item if $ u^{\xi}_{n, i} : X^{\xi}_{n} \to \G $ is a $ 2^{-n} $-isometry, then Lemma~\ref{lemmkubsol} provides $ j \in \mathbb{N} $ such that $ u^{\xi}_{n+1, j} : X^{\xi}_{n+1} \to \G $ is a $ 2^{-(n+1)} $-isometry and $ \Vert u^{\xi}_{n+1, j}|_{X^{\xi}_{n}} - u^{\xi}_{n, i} \Vert < 2 \cdot 2^{-n} $.
\end{itemize}

Using the assumption of the lemma, it is straightforward to show that the sets
\begin{align*}
A_{j} = \big\{ \xi \in \Xi : \; & \textrm{$ u^{\xi}_{1, j} $ is a $ 2^{-1} $-isometry} \big\}, \\
B_{n,i,j} = \big\{ \xi \in \Xi : \; & \textrm{$ u^{\xi}_{n+1, j} $ is a $ 2^{-(n+1)} $-isometry} \\
 & \textrm{and $ \Vert u^{\xi}_{n+1, j}|_{X^{\xi}_{n}} - u^{\xi}_{n, i} \Vert < 2 \cdot 2^{-n} $} \big\},
\end{align*}
are Borel for every $ n, i, j \in \mathbb{N} $.

We define recursively Borel functions $ j_{n} : \Xi \to \mathbb{N}, n = 1, 2, \dots, $ as follows. These functions are required to satisfy

(i) $ u^{\xi}_{n,j_{n}(\xi)} : X^{\xi}_{n} \to \G $ is a $ 2^{-n} $-isometry,

(ii) $ \Vert u^{\xi}_{n+1,j_{n+1}(\xi)}|_{X^{\xi}_{n}} - u^{\xi}_{n,j_{n}(\xi)} \Vert < 2 \cdot 2^{-n} $. 

Let $ j_{1}(\xi) $ be the least natural number $ j $ such that $ u^{\xi}_{1, j} : X^{\xi}_{1} \to \G $ is a $ 2^{-1} $-isometry. We already know that such a number exists. The function $ j_{1} $ is Borel, as
$$ (j_{1})^{-1}(\{ j \}) = A_{j} \setminus \bigcup_{l=1}^{j-1} A_{l}. $$
Assuming that $ j_{n}(\xi) $ is defined, let $ j_{n+1}(\xi) $ be the least natural number $ j $ such that $ u^{\xi}_{n+1, j} : X^{\xi}_{n+1} \to \G $ is a $ 2^{-(n+1)} $-isometry and $ \Vert u^{\xi}_{n+1, j}|_{X^{\xi}_{n}} - u^{\xi}_{n, j_{n}(\xi)} \Vert < 2 \cdot 2^{-n} $. We already know that such a number exists. The function $ j_{n+1} $ is Borel, as
$$ (j_{n+1})^{-1}(\{ j \}) = \bigcup_{i=1}^{\infty} \bigg[ (j_{n})^{-1}(\{ i \}) \cap \Big( B_{n,i,j} \setminus \bigcup_{l=1}^{j-1} B_{n,i,l} \Big) \bigg]. $$

Our next step is to define an isometry $ f_{\xi} : X_{\xi} \to \G $. For $ x \in X^{\xi}_{m} $, the sequence $ \{ u^{\xi}_{n,j_{n}(\xi)}(x) \}_{n \geq m} $ is Cauchy, since $ \Vert u^{\xi}_{n+1,j_{n+1}(\xi)}(x) - u^{\xi}_{n,j_{n}(\xi)}(x) \Vert \leq 2 \cdot 2^{-n} \Vert x \Vert_{\xi} $ for $ n \geq m $ by (ii). Therefore, we can put
$$ f_{\xi}(x) = \lim_{n\to\infty} u^{\xi}_{n,j_{n}(\xi)}(x), \quad x \in \bigcup_{m=1}^{\infty} X^{\xi}_{m}. $$
Using (i), we obtain
$$ (1 + 2^{-n})^{-1} \Vert x \Vert_{\xi} \leq \Vert u^{\xi}_{n,j_{n}(\xi)}(x) \Vert \leq (1 + 2^{-n}) \Vert x \Vert_{\xi}, \quad x \in X^{\xi}_{m}, \; n \geq m. $$
It follows that $ \Vert f_{\xi}(x) \Vert = \Vert x \Vert_{\xi} $ for $ x \in \bigcup_{m=1}^{\infty} X^{\xi}_{m} $ and so that there is a unique extension $ f_{\xi} : X_{\xi} \to \G $ satisfying $ \Vert f_{\xi}(x) \Vert = \Vert x \Vert_{\xi} $ for every $ x \in X_{\xi} $.

Let us realize that the mapping
$$ \chi_{k} : \Xi \to \G, \quad \chi_{k}(\xi) = f_{\xi}(x^{\xi}_{k}), $$
is Borel for every $ k \in \mathbb{N} $. Since
$$ \chi_{k}(\xi) = f_{\xi}(x^{\xi}_{k}) = \lim_{n\to\infty} u^{\xi}_{n,j_{n}(\xi)}(x^{\xi}_{k}) = \lim_{n\to\infty} u_{n,j_{n}(\xi)}(0, \dots, 0, \underset{k}{1}, 0, \dots, \underset{n}{0}), $$
the mapping $ \chi_{k} $ is the pointwise limit of a sequence of Borel mappings.

Finally, let us define the desired mapping $ \sel $. We may suppose that $ \G $ is a subspace of $ C([0, 1]) $. This allows us to define
$$ \sel : \Xi \to \mathcal{SE}(C([0, 1])), \quad \sel(\xi) = f_{\xi}(X_{\xi}), \quad \xi \in \Xi. $$
Since $ \sel $ fulfills the formula
$$ \sel(\xi) = \overline{\mathrm{span}} \, \{ \chi_{1}(\xi), \chi_{2}(\xi), \dots \}, $$
it is straightforward to show that it is a Borel mapping.
\end{proof}

\section{Tsirelson type spaces} \label{sec:tsirelson}

In this section, we will use Tsirelson type spaces introduced by S.~A.~Argyros and I.~Deliyanni \cite{argdel} to show that the class of spaces embeddable into $ c_{0} $ is not Borel (Theorem~\ref{thmc0subsp}). Those spaces are obtained by a generalization of the notion of an admissible family.

In fact, our approach is slightly different from the approach of Argyros and Deliyanni. The space defined below is derived from Tsirelson's original example $ \Tss $, not from its dual $ \Ts $ (some comments on spaces derived from $ \Ts $ are provided in Remark~\ref{remTs}). Moreover, we consider even more general systems of admissible families, including systems which lead to spaces quite different from $ \Tss $ (see Lemma~\ref{lemmaembc0}). In spite of this, for our purposes, we use the symbol $ \Tss $ also for these non Tsirelson-like spaces.

Throughout this paper, we identify elements of $ 2^{\mathbb{N}} $ with subsets of $ \mathbb{N} $. For this reason, members of $ \mathcal{K}(2^{\mathbb{N}}) $ represent systems of subsets of $ \mathbb{N} $.

Let $ e_{1}, e_{2}, \dots $ be the canonical basis of $ c_{00} $ (i.e., $ e_{n} = \mathbf{1}_{\{ n \}} $). Let us recall that we denote $ Ex = \mathbf{1}_{E} \cdot x $ for $ E \subset \mathbb{N} $ and $ x \in c_{00} $.

For $ \mathcal{M} \in \mathcal{K}(2^{\mathbb{N}}) $, a family $ \{ E_{1}, \dots, E_{n} \} $ of successive finite subsets of $ \mathbb{N} $ is said to be \emph{$ \mathcal{M} $-admissible} if an element of $ \mathcal{M} $ contains numbers $ m_{1}, \dots, m_{n} $ such that
$$ m_{1} \leq E_{1} < m_{2} \leq E_{2} < \dots < m_{n} \leq E_{n}. $$
The system of all $ \mathcal{M} $-admissible families is denoted by $ \mathrm{adm}(\mathcal{M}) $.

\begin{definition}
For $ \mathcal{M} \in \mathcal{K}(2^{\mathbb{N}}) $, let $ \Theta_\mathcal{M} $ be the smallest absolutely convex subset of $ c_{00} $ containing every $ e_{i}, i \in \mathbb{N}, $ and satisfying the property
$$ \{ E_{1}, \dots, E_{n} \} \in \mathrm{adm}(\mathcal{M}) \; \& \; x_{1}, \dots, x_{n} \in \Theta_\mathcal{M} \quad \Rightarrow \quad \frac{1}{2} \sum_{k=1}^{n} E_{k}x_{k} \in \Theta_\mathcal{M}. $$
Let $ \Vert \cdot \Vert_\mathcal{M} $ be the Minkowski gauge of $ \Theta_\mathcal{M} $ and let $ \Tss[\mathcal{M}, \frac{1}{2}] $ be a completion of $ (c_{00}, \Vert \cdot \Vert_\mathcal{M}) $.
\end{definition}

First, we introduce without proof some simple facts about $ \Tss[\mathcal{M}, \frac{1}{2}] $.

\begin{fact} \label{fact1uncon}
The sequence $ e_{1}, e_{2}, \dots $ is a $ 1 $-unconditional basis of $ \Tss[\mathcal{M}, \frac{1}{2}] $.
\end{fact}

\begin{fact} \label{factsupineq}
If $ \{ E_{1}, \dots, E_{n} \} $ is $ \mathcal{M} $-admissible and $ x_{1}, \dots, x_{n} \in \Tss[\mathcal{M}, \frac{1}{2}] $, then
$$ \Big\Vert \sum_{k=1}^{n} E_{k}x_{k} \Big\Vert_\mathcal{M} \leq 2 \sup_{1 \leq k \leq n} \Vert x_{k} \Vert_\mathcal{M}. $$
\end{fact}

\begin{fact} \label{factTsscontin}
If $ \mathcal{M}_{1}, \mathcal{M}_{2} \in \mathcal{K}(2^{\mathbb{N}}) $ are two systems such that
$$ \big\{ A \cap \{ 1, \dots, \ell \} : A \in \mathcal{M}_{1} \big\} = \big\{ A \cap \{ 1, \dots, \ell \} : A \in \mathcal{M}_{2} \big\}, $$
then $ \Vert x \Vert_{\mathcal{M}_{1}} = \Vert x \Vert_{\mathcal{M}_{2}} $ for every $ x \in \mathrm{span} \{ e_{1}, e_{2}, \dots, e_{\ell} \} $.
\end{fact}

The following important property of $ \Tss[\mathcal{M}, \frac{1}{2}] $ follows from the proof of \cite[Proposition~1.1]{argdel}.

\begin{lemma}[Argyros, Deliyanni] \label{lemmaargdel}
If $ \mathcal{M} $ consists of finite sets only, then $ e_{1}, e_{2}, \dots $ is a boundedly complete basis of $ \Tss[\mathcal{M}, \frac{1}{2}] $, and so $ \Tss[\mathcal{M}, \frac{1}{2}] $ can not be embedded isomorphically into $ c_{0} $.
\end{lemma}

We are going to show that the space $ \Tss[\mathcal{M}, \frac{1}{2}] $ is considerably different from $ \Tss $ if $ \mathcal{M} $ contains an infinite set. This phenomenon is the heart of our argument.

\begin{lemma} \label{lemmaembc0}
If $ \mathcal{M} $ contains an infinite set, then $ \Tss[\mathcal{M}, \frac{1}{2}] $ is isomorphic to the $ c_{0} $-sum of a sequence of finite-dimensional spaces, and so $ \Tss[\mathcal{M}, \frac{1}{2}] $ can be embedded isomorphically into $ c_{0} $.
\end{lemma}

\begin{proof}
Assuming that $ \{ m_{1} < m_{2} < \dots \} \in \mathcal{M} $, we show for every $ x \in \Tss[\mathcal{M}, \frac{1}{2}] $ that
$$ \sup_{k\in\mathbb{N} \cup \{ 0 \}} \Vert E_{k}x \Vert_\mathcal{M} \leq \Vert x \Vert_\mathcal{M} \leq \Vert E_{0}x \Vert_\mathcal{M} + 2 \sup_{k\in\mathbb{N}} \Vert E_{k}x \Vert_\mathcal{M} $$
where $ E_{0} = \{ 1, \dots, m_{1} - 1 \} $ and $ E_{k} = \{ m_{k}, \dots , m_{k+1} - 1 \} $. The first inequality follows from Fact~\ref{fact1uncon}. For $ n \in \mathbb{N} $, the family $ \{ E_{1}, \dots, E_{n} \} $ is $ \mathcal{M} $-admissible, and thus
$$ \Big\Vert \sum_{k=0}^{n} E_{k}x \Big\Vert_\mathcal{M} \leq \Vert E_{0}x \Vert_\mathcal{M} + \Big\Vert \sum_{k=1}^{n} E_{k}x \Big\Vert_\mathcal{M} \leq \Vert E_{0}x \Vert_\mathcal{M} + 2 \sup_{1 \leq k \leq n} \Vert E_{k}x \Vert_\mathcal{M} $$
by Fact~\ref{factsupineq}. Since $ x = \sum_{k=0}^{\infty} E_{k}x $, the remaining inequality follows.
\end{proof}

The following lemma, proof of which is essentially contained in \cite{argdel}, will be useful later.

\begin{lemma} \label{lemmshrink}
If $ \mathcal{M} $ contains all three-element sets, then $ e_{1}, e_{2}, \dots $ is a shrinking basis of $ \Tss[\mathcal{M}, \frac{1}{2}] $. In particular, if $ \mathcal{M} $ contains all three-element sets but it consists of finite sets only, then $ \Tss[\mathcal{M}, \frac{1}{2}] $ is reflexive.
\end{lemma}

\begin{proof}
Let $ e^{*}_{1}, e^{*}_{2}, \dots $ be the dual basic sequence. Given $ x^{*} \in (\Tss[\mathcal{M}, \frac{1}{2}])^{*} $, we want to show that
$$ x^{*} = \sum_{i=1}^{\infty} x^{*}(e_{i}) e^{*}_{i}. $$
Assume the opposite, i.e., that
$$ \varepsilon = \lim_{n \to \infty} \Big\Vert x^{*} - \sum_{i=1}^{n-1} x^{*}(e_{i}) e^{*}_{i} \Big\Vert_\mathcal{M} > 0 $$
(we note that the sequence under the limit is non-increasing, due to $ 1 $-unconditionality). Choose $ m_{1} \in \mathbb{N} $ so that
$$ \Big\Vert x^{*} - \sum_{i=1}^{m_{1}-1} x^{*}(e_{i}) e^{*}_{i} \Big\Vert_\mathcal{M} < \frac{4}{3} \, \varepsilon $$
and $ m_{2}, m_{3}, m_{4} $ so that
$$ \Big\Vert \sum_{i=m_{k}}^{m_{k+1}-1} x^{*}(e_{i}) e^{*}_{i} \Big\Vert_\mathcal{M} > \frac{8}{9} \, \varepsilon, \quad k = 1, 2, 3. $$
For $ k = 1, 2, 3 $, if we put $ E_{k} = \{ m_{k}, \dots , m_{k+1} - 1 \} $, then we can find $ x_{k} \in \Tss[\mathcal{M}, \frac{1}{2}], \Vert x_{k} \Vert_\mathcal{M} \leq 1, $ such that $ x^{*}(E_{k}x_{k}) > \frac{8}{9} \varepsilon $. By our assumption, $ \mathcal{M} $ contains $ \{ m_{1}, m_{2}, m_{3} \} $, and so the family $ \{ E_{1}, E_{2}, E_{3} \} $ is $ \mathcal{M} $-admissible. Hence, $ \Vert E_{1}x_{1} + E_{2}x_{2} + E_{3}x_{3} \Vert_\mathcal{M} \leq 2 $ by Fact~\ref{factsupineq}, and we obtain
$$ 3 \cdot \frac{8}{9} \; \varepsilon < x^{*}(E_{1}x_{1} + E_{2}x_{2} + E_{3}x_{3}) < \frac{4}{3} \, \varepsilon \cdot \Vert E_{1}x_{1} + E_{2}x_{2} + E_{3}x_{3} \Vert_\mathcal{M} \leq \frac{4}{3} \, \varepsilon \cdot 2, $$
a contradiction.

The second part of the statement follows from Lemma~\ref{lemmaargdel} and Theorem~\ref{thmjames}.
\end{proof}

\begin{lemma} \label{lemmselectTss}
There exists a Borel mapping $ \sel : \mathcal{K}(2^{\mathbb{N}}) \to \mathcal{SE}(C([0, 1])) $ such that $ \sel(\mathcal{M}) $ is isometric to $ \Tss[\mathcal{M}, \frac{1}{2}] $ for every $ \mathcal{M} \in \mathcal{K}(2^{\mathbb{N}}) $.
\end{lemma}

\begin{proof}
Due to Lemma~\ref{lemmselect}, it is sufficient to realize that the function $ \mathcal{M} \mapsto \Vert x \Vert_\mathcal{M} $ is Borel for every $ x \in c_{00} $. We show that this function is continuous. By Fact~\ref{factTsscontin}, if $ \ell $ is such that $ x \in \mathrm{span} \{ e_{1}, \dots, e_{\ell} \} $, then the norm $ \Vert x \Vert_\mathcal{M} $ depends only on $ \{ A \cap \{ 1, \dots, \ell \} : A \in \mathcal{M} \} $. For this reason, $ \mathcal{K}(2^{\mathbb{N}}) $ can be decomposed into finitely many clopen sets on which $ \Vert x \Vert_\mathcal{M} $ is constant.
\end{proof}

\begin{proof}[Proof of Theorem \ref{thmc0subsp}]
It is easy to show that the class of all Banach spaces $ X $ which can be embedded isomorphically into $ c_{0} $ (shortly $ X \hookrightarrow c_{0} $) is analytic (see \cite[Theorem~2.3]{bossard}). Let us show that it is hard analytic. The set of all infinite subsets of $ \mathbb{N} $ is a $ G_{\delta} $ but not $ F_{\sigma} $ subset of $ 2^{\mathbb{N}} $. By Theorem~\ref{thmhur}, the set
$$ C = \big\{ \mathcal{M} \in \mathcal{K}(2^{\mathbb{N}}) : \mathcal{M} \textrm{ contains an infinite set} \big\} $$
is complete analytic. Let $ \sel $ be a mapping provided by Lemma~\ref{lemmselectTss}. Using Lemma~\ref{lemmaargdel} and Lemma~\ref{lemmaembc0}, we obtain
$$ \sel(\mathcal{M}) \hookrightarrow c_{0} \quad \Leftrightarrow \quad \mathcal{M} \in C. $$
It remains to apply Lemma~\ref{lemmhardset}.
\end{proof}

\begin{remark} \label{remTss}
(i) The space $ c_{0} $ is not the only example for which the argument works. If a separable Banach space $ Z $ contains an isomorphic copy of $ c_{0} $ but does not contain an infinite-dimensional reflexive subspace, then the class of all Banach spaces which can be embedded isomorphically into $ Z $ is complete analytic.

(ii) Let $ G_{1}, G_{2}, \dots $ be a dense sequence of finite-dimensional spaces (i.e., for every finite-dimensional Banach space $ G $ and every $ \varepsilon > 0 $, there is a bijective $ \varepsilon $-isometry between $ G $ and some $ G_{n} $). Then the class of all spaces isomorphic to $ (\bigoplus G_{n})_{c_{0}} $ is complete analytic. Indeed, the space $ \Tss[\mathcal{M}, \frac{1}{2}] \oplus (\bigoplus G_{n})_{c_{0}} $ is isomorphic to $ (\bigoplus G_{n})_{c_{0}} $ if and only if $ \mathcal{M} $ contains an infinite set.
\end{remark}

\begin{remark} \label{remTs}
Argyros and Deliyanni \cite{argdel} defined a space $ \Ts[\mathcal{M}, \frac{1}{2}] $ as the Banach space of sequences $ x = \{ x(i) \}_{i=1}^{\infty} $ with the basis $ e_{i} = \mathbf{1}_{\{ i \}} $ and with the implicitly defined norm
$$ \Vert x \Vert_\mathcal{M} = \max \Bigg\{ \Vert x \Vert_{\infty}, \frac{1}{2} \sup \bigg\{ \sum_{k=1}^{n} \Vert E_{k}x \Vert_\mathcal{M} : \{ E_{1}, \dots, E_{n} \} \in \mathrm{adm}(\mathcal{M}) \bigg\} \Bigg\}. $$
It can be shown that the sequence $ e_{i} $ considered as a basis of $ \Tss[\mathcal{M}, \frac{1}{2}] $ is dual to the same sequence $ e_{i} $ considered as a basis of $ \Ts[\mathcal{M}, \frac{1}{2}] $. However, the duality between these two spaces is not warranted in our general setting.

Let us mention some notes and consequences of results from this section.

(i) If $ \mathcal{M} $ consists of finite sets only, then $ e_{1}, e_{2}, \dots $ is a shrinking basis of $ \Ts[\mathcal{M}, \frac{1}{2}] $. In this case, $ \Tss[\mathcal{M}, \frac{1}{2}] $ is the dual of $ \Ts[\mathcal{M}, \frac{1}{2}] $.

(ii) If $ \mathcal{M} $ contains an infinite set, then $ \Ts[\mathcal{M}, \frac{1}{2}] $ is isomorphic to the $ \ell_{1} $-sum of a sequence of finite-dimensional spaces. In this case, the dual of $ \Ts[\mathcal{M}, \frac{1}{2}] $ is not separable and contains $ \Tss[\mathcal{M}, \frac{1}{2}] $ as a proper subspace. Notice also that $ \Ts[\mathcal{M}, \frac{1}{2}] $ has the Schur property.

(iii) If $ \mathcal{M} $ contains all three-element sets, then $ e_{1}, e_{2}, \dots $ is a boundedly complete basis of $ \Ts[\mathcal{M}, \frac{1}{2}] $, and $ \Ts[\mathcal{M}, \frac{1}{2}] $ is the dual of $ \Tss[\mathcal{M}, \frac{1}{2}] $.

(iv) If $ \mathcal{M} $ contains all three-element sets but it consists of finite sets only, then $ \Ts[\mathcal{M}, \frac{1}{2}] $ is reflexive, as well as $ \Tss[\mathcal{M}, \frac{1}{2}] $.

(v) There exists a Borel mapping $ \sel : \mathcal{K}(2^{\mathbb{N}}) \to \mathcal{SE}(C([0, 1])) $ such that $ \sel(\mathcal{M}) $ is isometric to $ \Ts[\mathcal{M}, \frac{1}{2}] $ for every $ \mathcal{M} \in \mathcal{K}(2^{\mathbb{N}}) $.

(vi) Let us denote by $ \mathbb{N}^{[\leq 3]} $ the system of all subsets of $ \mathbb{N} $ with at most three elements. Then $ \mathcal{M} \mapsto \sel(\mathcal{M} \cup \mathbb{N}^{[\leq 3]}) $ is a Borel mapping which maps a complete analytic set into spaces with the Schur property and its complement into reflexive spaces. Therefore, it follows from the Tsirelson space method that the class of all separable Banach spaces with the Schur property is not coanalytic. Using a tree space method developed in \cite{bossard}, it can be shown that this class is not analytic (see \cite[Theorem~27]{braga}). Our proof of the proper complexity result (Theorem~\ref{thmschur}) can be considered as a combination of these two methods.

(vii) The class of all spaces isomorphic to $ (\bigoplus G_{n})_{\ell_{1}} $ is complete analytic. Indeed, the space $ \Ts[\mathcal{M}, \frac{1}{2}] \oplus (\bigoplus G_{n})_{\ell_{1}} $ is isomorphic to $ (\bigoplus G_{n})_{\ell_{1}} $ if and only if $ \mathcal{M} $ contains an infinite set.
\end{remark}

\section{Preliminaries II} \label{sec:prelimII}

By $ \Lambda^{< \mathbb{N}} $ we denote the system of all finite sequences of elements of a set $ \Lambda $, including the empty sequence $ \emptyset $. That is,
$$ \Lambda^{< \mathbb{N}} = \bigcup _{\ell=0}^{\infty } \Lambda^{\ell} $$
where $ \Lambda^{0} = \{ \emptyset \} $. By $ |\eta| $ we mean the length of $ \eta \in \Lambda^{< \mathbb{N}} $. For $ \sigma \in \Lambda^{\mathbb{N}} $, we denote by $ \sigma|_{\ell} $ its initial segment $ (\sigma(1), \dots, \sigma(\ell)) $ of length $ \ell \in \mathbb{N} $.

A subset $ T $ of $ \Lambda^{< \mathbb{N}} $ is called a \emph{tree on $ \Lambda $} if it is downward closed, i.e.,
$$ (\lambda_{1}, \lambda_{2}, \dots, \lambda_{k}) \in T \; \& \; j \leq k \quad \Rightarrow \quad (\lambda_{1}, \lambda_{2}, \dots, \lambda_{j}) \in T. $$
The set of all trees on $ \Lambda $ is denoted by $ \mathrm{Tr}(\Lambda) $ and endowed with the topology induced by the topology of $ 2^{\Lambda^{< \mathbb{N}}} $. The set of all infinite branches of $ T \in \mathrm{Tr}(\Lambda) $, i.e., sequences $ \nu \in \Lambda^{\mathbb{N}} $ such that $ T $ contains all initial segments of $ \nu $, is denoted by $ [T] $.

In what follows, we identify $ (\Theta \times \Lambda)^{\ell} $ with $ \Theta^{\ell} \times \Lambda^{\ell} $ and $ (\Theta \times \Lambda)^{\mathbb{N}} $ with $ \Theta^{\mathbb{N}} \times \Lambda^{\mathbb{N}} $. In this way, elements of a tree on $ \Theta \times \Lambda $ are pairs of sequences of the same length and its infinite branches are elements of $ \Theta^{\mathbb{N}} \times \Lambda^{\mathbb{N}} $.

If $ \T $ is a tree on $ \Theta \times \Lambda $ and $ \sigma \in \Theta^{\mathbb{N}} $, we define
$$ \T(\sigma) \in \mathrm{Tr}(\Lambda), \quad \T(\sigma) = \{ \nu \in \Lambda^{< \mathbb{N}} : (\sigma|_{|\nu|} , \nu) \in \T \}. $$

We say that a tree $ T $ on $ \mathbb{N} $ is \emph{ill-founded} ($ T \in \mathrm{IF} $) if it has an infinite branch (i.e., $ [T] \neq \emptyset $). In the opposite case, we say that $ T $ is \emph{well-founded} ($ T \in \mathrm{WF} $).

\begin{lemma} \label{lemmCPCA}
The set
$$ C = \big\{ \T \in \mathrm{Tr}(2 \times \mathbb{N}) : (\forall \sigma \in 2^{\mathbb{N}})(\T(\sigma) \in \mathrm{IF}) \big\} $$
is a $ \mathbf{\Pi}^{1}_{2} $-complete subset of {} $ \mathrm{Tr}(2 \times \mathbb{N}) $.
\end{lemma}

It is easy (and not necessary for our purposes actually) to check that $ C $ is a $ \mathbf{\Pi}^{1}_{2} $ set. To prove that it is $ \mathbf{\Pi}^{1}_{2} $-hard, we will use the following well-known results:
\begin{itemize}
\item all uncountable standard Borel spaces are Borel isomorphic (see e.g. \cite[(15.6)]{kechris}) and, for this reason, it is possible to consider only $ Y = \mathbb{N}^{\mathbb{N}} $ in the definition of a $ \Gamma $-hard set (where $ \Gamma = \mathbf{\Pi}^{1}_{2} $) and only $ Y = 2^{\mathbb{N}} $ in the definition of a $ \mathbf{\Sigma}^{1}_{2} $ set,
\item a subset $ A $ of a Polish space $ X $ is analytic if and only if it is the projection of some closed $ F \subset X \times \mathbb{N}^{\mathbb{N}} $ (see e.g. \cite[(14.3)]{kechris}),
\item for a closed $ F \subset \Lambda^{\mathbb{N}} $, we have $ F = [T] $ for some $ T \in \mathrm{Tr}(\Lambda) $ (it is possible to collect all initial segments of elements of $ F $, cf. \cite[(2.4)]{kechris}).
\end{itemize}

\begin{proof}(cf. with \cite[(37.11)]{kechris}).
Let $ A $ be a $ \mathbf{\Pi}^{1}_{2} $ subset of $ \mathbb{N}^{\mathbb{N}} $. We need to find a Borel mapping $ f : \mathbb{N}^{\mathbb{N}} \to \mathrm{Tr}(2 \times \mathbb{N}) $ such that $ f^{-1}(C) = A $. There exists an analytic subset $ B $ of $ \mathbb{N}^{\mathbb{N}} \times 2^{\mathbb{N}} $ such that $ \mathbb{N}^{\mathbb{N}} \setminus A = \mathrm{proj}_{\mathbb{N}^{\mathbb{N}}} ((\mathbb{N}^{\mathbb{N}} \times 2^{\mathbb{N}}) \setminus B) $, i.e.,
$$ \nu \in A \quad \Leftrightarrow \quad \forall \sigma \in 2^{\mathbb{N}} : (\nu, \sigma) \in B $$
for all $ \nu \in \mathbb{N}^{\mathbb{N}} $. There exists a closed subset $ F $ of $ \mathbb{N}^{\mathbb{N}} \times 2^{\mathbb{N}} \times \mathbb{N}^{\mathbb{N}} $ such that $ B = \mathrm{proj}_{\mathbb{N}^{\mathbb{N}} \times 2^{\mathbb{N}}} F $, i.e.,
$$ (\nu, \sigma) \in B \quad \Leftrightarrow \quad \exists \omega \in \mathbb{N}^{\mathbb{N}} : (\nu, \sigma, \omega) \in F. $$
There is a tree $ \mathfrak{T} \in \mathrm{Tr}(\mathbb{N} \times 2 \times \mathbb{N}) $ such that $ [\mathfrak{T}] = F $. We have
\begin{align*}
\nu \in A \; & \Leftrightarrow \; \forall \sigma \in 2^{\mathbb{N}} \; \exists \omega \in \mathbb{N}^{\mathbb{N}} : (\nu, \sigma, \omega) \in [\mathfrak{T}] \\
& \Leftrightarrow \; \forall \sigma \in 2^{\mathbb{N}} \; \exists \omega \in \mathbb{N}^{\mathbb{N}} : \omega \in [\mathfrak{T}(\nu, \sigma)] \\
& \Leftrightarrow \; \forall \sigma \in 2^{\mathbb{N}} : \mathfrak{T}(\nu, \sigma) \in \mathrm{IF} \\
& \Leftrightarrow \; \mathfrak{T}(\nu) \in C
\end{align*}
for all $ \nu \in \mathbb{N}^{\mathbb{N}} $, and it remains to note that the mapping $ \nu \mapsto \mathfrak{T}(\nu) $ is continuous.
\end{proof}

A bounded linear operator $ T : X \to Y $ is called \emph{weakly compact} if the image of the unit ball of $ X $ is relatively weakly compact in $ Y $. The operator $ T $ is called \emph{completely continuous} if it maps weakly convergent sequences to norm convergent ones.

We say that a Banach space $ X $ has the \emph{Dunford-Pettis property} if every weakly compact operator $ T : X \to Y $ from $ X $ into another Banach space $ Y $ is completely continuous.

In the remainder of this section, we prove the easy part of Theorem~\ref{thmschur}.

\begin{lemma} \label{lemmDP}
The class of all separable Banach spaces with the Dunford-Pettis property is $ \mathbf{\Pi}^{1}_{2} $.
\end{lemma}

During the proof, we will use the following known facts:
\begin{itemize}
\item an operator $ T : X \to Y $ is completely continuous if and only if it maps weakly null sequences to null sequences,
\item $ X $ has the Dunford-Pettis property if and only if every weakly compact operator $ T : X \to c_{0} $ is completely continuous (see e.g. \cite[Theorem~1]{diestel}).
\end{itemize}

\begin{proof}
We prove first that $ X \in \mathcal{SE}(C([0, 1])) $ has the Dunford-Pettis property if and only if
$$ \forall (x_{1}, x_{2}, \dots) \in B_{C[0, 1]}^{\mathbb{N}} \; \forall (y_{1}, y_{2}, \dots) \in B_{c_{0}}^{\mathbb{N}} : \textrm{(a) or (b) or (c) or (d) or (e),} $$
where we consider properties
\begin{itemize}
\item[{(a)}] $ B_{X} \neq \overline{\{ x_{1}, x_{2}, \dots \}} $,
\item[{(b)}] $ x_{n} \mapsto y_{n} $ does not define a bounded linear operator from $ \overline{\mathrm{span}} \{ x_{1}, x_{2}, \dots \} $ into $ c_{0} $,
\item[{(c)}] $ \{ y_{1}, y_{2}, \dots \} $ is not relatively weakly compact,
\item[{(d)}] $ x_{2}, x_{4}, x_{6}, \dots $ is not weakly null,
\item[{(e)}] $ y_{2}, y_{4}, y_{6}, \dots $ is null.
\end{itemize}

Let us assume that $ X $ has the Dunford-Pettis property. For sequences $ x_{1}, x_{2}, \dots $ in $ B_{C[0, 1]} $ and $ y_{1}, y_{2}, \dots $ in $ B_{c_{0}} $, we need to show that some of the five properties is satisfied. Let us suppose that none of properties (a), (b), (c), (d) is satisfied. So, $ x_{n} \mapsto y_{n} $ defines a bounded linear operator $ T : X \to c_{0} $ that is weakly compact. Moreover, the sequence $ x_{2}, x_{4}, x_{6}, \dots $ is weakly null. As $ X $ is assumed to have the Dunford-Pettis property, $ T $ is completely continuous, and thus it maps the weakly null sequence $ x_{2k} $ to a null sequence $ y_{2k} $. It means that (e) is valid.

Let us assume that the formula is fulfilled for some $ X \in \mathcal{SE}(C([0, 1])) $. Let $ T : X \to c_{0} $ be a weakly compact operator. We need to show that $ T $ maps a weakly null sequence $ a_{1}, a_{2}, \dots $ to a null sequence. We may suppose that $ \Vert T \Vert \leq 1 $ and that $ a_{k} \in B_{X} $. Let us put $ x_{2k} = a_{k} $ and choose a sequence $ x_{1}, x_{3}, \dots $ that is dense in $ B_{X} $. Moreover, let $ y_{n} = Tx_{n} $. As $ T $ is weakly compact, the set $ \{ y_{1}, y_{2}, \dots \} $ is relatively weakly compact. So, none of properties (a), (b), (c), (d) is satisfied. Then (e) has to be valid. That is, the sequence $ Ta_{k} = y_{2k} $ is null.

So, both implications are verified. To prove the lemma, it remains to show that each of the five properties define an analytic subset of $ \mathcal{SE}(C([0, 1])) \times B_{C[0, 1]}^{\mathbb{N}} \times B_{c_{0}}^{\mathbb{N}} $.

(a) The corresponding set is Borel. Indeed, if $ U $ is the open unit ball of $ C[0, 1] $ and $ G_{1}, G_{2}, \dots $ is a basis of the norm topology of $ C[0, 1] $, then $ B_{X} = \overline{\{ x_{1}, x_{2}, \dots \}} $ if and only if
$$ \forall k \in \mathbb{N} : [ X \cap U \cap G_{k} \neq \emptyset \Leftrightarrow \exists n \in \mathbb{N} : x_{n} \in G_{k} ]. $$

(b) The corresponding set is Borel. It is sufficient to realize that $ x_{n} \mapsto y_{n} $ defines a bounded linear operator from $ \overline{\mathrm{span}} \{ x_{1}, x_{2}, \dots \} $ into $ c_{0} $ if and only if
$$ \exists K \in \mathbb{N} \forall m \in \mathbb{N} \forall \alpha_{1}, \alpha_{2}, \dots, \alpha_{m} \in \mathbb{Q} : \Big\Vert \sum_{n=1}^{m} \alpha_{n} y_{n} \Big\Vert \leq K \Big\Vert \sum_{n=1}^{m} \alpha_{n} x_{n} \Big\Vert. $$

(c) Let us notice that $ B_{c_{0}} $ with the weak topology is a subspace of the topological product $ [-1, 1]^{\mathbb{N}} $. A subset of $ B_{c_{0}} $ is relatively weakly compact if and only if its closure in $ [-1, 1]^{\mathbb{N}} $ is still a subset of $ B_{c_{0}} $. For this reason, $ \{ y_{1}, y_{2}, \dots \} $ is not relatively weakly compact if and only if
\begin{align*}
\exists z \in [-1, 1]^{\mathbb{N}} : & \; \big[ \exists l \in \mathbb{N} \forall m \in \mathbb{N} \exists k \geq m : |z(k)| \geq 1/l \big] \\
 & \; \& \big[ \forall l, m \in \mathbb{N} \exists n \in \mathbb{N} \forall k \leq m : |y_{n}(k) - z(k)| < 1/l \big] .
\end{align*}
Hence, our set is a projection of a Borel subset of $ \mathcal{SE}(C([0, 1])) \times B_{C[0, 1]}^{\mathbb{N}} \times B_{c_{0}}^{\mathbb{N}} \times [-1, 1]^{\mathbb{N}} $.

(d) The corresponding set is analytic by \cite[Theorem~20]{braga}.

(e) It is easy to show that the corresponding set is Borel.
\end{proof}

\section{Tree spaces upon Tsirelson spaces}

In this section, we apply the construction of a tree space studied in \cite{kurka} on Tsirelson type spaces presented above. This will enable us to show that some classes of Banach spaces have quite high complexity (Theorem~\ref{thmschur}).

For a finite sequence $ \nu = (n_{1}, n_{2}, \dots, n_{k}) \in \mathbb{N}^{< \mathbb{N}} $, let $ \tilde{\nu} = \{ n_{1} < n_{1} + n_{2} < \dots < \sum_{i=1}^{k} n_{i} \} \subset \mathbb{N} $. Similarly, for an infinite sequence $ \nu = (n_{1}, n_{2}, \dots) \in \mathbb{N}^{\mathbb{N}} $, let $ \tilde{\nu} = \{ n_{1} < n_{1} + n_{2} < \dots \} \subset \mathbb{N} $.

For every $ T \in \mathrm{Tr}(\mathbb{N}) $, we define
$$ \mathcal{M}_{T} = \{ \tilde{\nu} : \nu \in T \cup [T] \textrm{ or } |\nu| \leq 3 \}. $$
Let us note that $ \mathcal{M}_{T} $ belongs to $ \mathcal{K}(2^{\mathbb{N}}) $ and that it contains an infinite set if and only if $ T $ is ill-founded. Thus, we obtain from Lemma~\ref{lemmaembc0} and Lemma~\ref{lemmshrink} a basic discovery about $ \Tss[\mathcal{M}_{T}, \frac{1}{2}] $.

\begin{lemma} \label{lemmTssMT}
{\rm (1)} If $ T \in \mathrm{Tr}(\mathbb{N}) $ is ill-founded, then $ \Tss[\mathcal{M}_{T}, \frac{1}{2}] $ is isomorphic to the $ c_{0} $-sum of a sequence of finite-dimensional spaces.

{\rm (2)} If $ T \in \mathrm{Tr}(\mathbb{N}) $ is well-founded, then $ \Tss[\mathcal{M}_{T}, \frac{1}{2}] $ is reflexive.
\end{lemma}

The following observation follows from Fact~\ref{factTsscontin}.

\begin{fact} \label{factMTcontin}
If $ T $ and $ S $ are two trees on $ \mathbb{N} $ which have the same sequences of length at most $ \ell $, then $ \Vert x \Vert_{\mathcal{M}_{T}} = \Vert x \Vert_{\mathcal{M}_{S}} $ for every $ x \in \mathrm{span} \{ e_{1}, e_{2}, \dots, e_{\ell} \} $.

In particular, if $ \T \in \mathrm{Tr}(2 \times \mathbb{N}) $ and $ \sigma, \tau \in 2^{\mathbb{N}} $ satisfy $ \sigma|_{\ell} = \tau|_{\ell} $, then $ \Vert x \Vert_{\mathcal{M}_{\T(\sigma)}} = \Vert x \Vert_{\mathcal{M}_{\T(\tau)}} $ for every $ x \in \mathrm{span} \{ e_{1}, e_{2}, \dots, e_{\ell} \} $.
\end{fact}

Now, we are ready to introduce our tree space.

\begin{definition}
For $ \T \in \mathrm{Tr}(2 \times \mathbb{N}) $, let $ E_{\T} $ be defined as a completion of $ c_{00}(2^{< \mathbb{N}} \setminus \{ \emptyset \}) $ with the norm
$$ \Vert x \Vert_{\T} = \sup_{\sigma \in 2^{\mathbb{N}}} \Big\Vert \sum_{\ell = 1}^{\infty} x(\sigma|_{\ell}) e_{\ell} \Big\Vert_{\mathcal{M}_{\T(\sigma)}}. $$ 
\end{definition}

This space is defined according to \cite[Definition 3.1]{kurka}. Indeed, we can take $ T = 2^{< \mathbb{N}} \setminus \{ \emptyset \}, F_{\sigma} = \Tss[\mathcal{M}_{\T(\sigma)}, \frac{1}{2}] $ and $ f^{\sigma}_{\ell} = e_{\ell} $ for $ \sigma \in 2^{\mathbb{N}} $ and $ \ell \in \mathbb{N} $. The requirement from \cite[Definition 3.1]{kurka} is satisfied due to Fact~\ref{factMTcontin}.

Thus, results from \cite[Section~3]{kurka} are available. In particular, if we denote by $ \{ z_{\eta} : \eta \in 2^{< \mathbb{N}} \setminus \{ \emptyset \} \} $ the canonical basis of $ c_{00}(2^{< \mathbb{N}} \setminus \{ \emptyset \}) $, then this system forms a basis of $ E_{\T} $.

The following two statements follow from \cite[Fact~3.2]{kurka}, \cite[Proposition~3.5]{kurka} and Lemma~\ref{lemmshrink}.

\begin{fact} \label{factbranches}
For every $ \sigma \in 2^{\mathbb{N}} $, we have the $ 1 $-equivalence
$$ \Big\Vert \sum_{\ell=1}^{n} \lambda_{\ell} z_{\sigma|_{\ell}} \Big\Vert_{\T} = \Big\Vert \sum_{\ell=1}^{n} \lambda_{\ell} e_{\ell} \Big\Vert_{\mathcal{M}_{\T(\sigma)}} $$
of the sequences $ z_{\sigma|_{1}}, z_{\sigma|_{2}}, \dots $ and $ e_{1}, e_{2}, \dots $ .

Therefore, $ E_{\T} $ contains a $ 1 $-complemented copy of $ \Tss[{\mathcal{M}_{\T(\sigma)}}, \frac{1}{2}] $. Consequently, $ E_{\T}^{*} $ contains a $ 1 $-complemented copy of the dual of $ \Tss[{\mathcal{M}_{\T(\sigma)}}, \frac{1}{2}] $.
\end{fact}

\begin{lemma} \label{lemmashrink}
The system $ \{ z_{\eta} : \eta \in 2^{< \mathbb{N}} \setminus \{ \emptyset \} \} $ is a $ 1 $-unconditional shrinking basis of $ E_{\T} $.
\end{lemma}

The following crucial lemma will be proven in a separate section.

\begin{lemma} \label{lemmal1subseq}
Let $ \T \in \mathrm{Tr}(2 \times \mathbb{N}) $ be such that $ \forall \sigma \in 2^{\mathbb{N}} : \T(\sigma) \in \mathrm{IF} $. If $ x^{*}_{1}, x^{*}_{2}, \dots $ is a normalized sequence in $ E_{\T}^{*} $ which converges to $ 0 $ in the $ w^{*} $-topology, then it has a subsequence $ y^{*}_{1}, y^{*}_{2}, \dots $ such that
$$ \Big\Vert \sum_{k=1}^{n} \lambda_{k} y^{*}_{k} \Big\Vert_{\T} \geq \frac{1}{5} \sum_{k=1}^{n} |\lambda_{k}| $$
for all $ n \in \mathbb{N} $ and $ \lambda_{1}, \dots, \lambda_{n} \in \mathbb{R} $.
\end{lemma}

It is not difficult to show that $ E_{\T}^{*} $ has the Schur property if it satisfies the conclusion of Lemma~\ref{lemmal1subseq}. Nevertheless, we show that a bit more can be said.

For a bounded sequence $ x_{1}, x_{2}, \dots $ in a Banach space $ X $, let us consider quantities
$$ \mathrm{ca}(x_{n}) = \inf_{m \in \mathbb{N}} \mathrm{diam} \{ x_{n} : n \geq m \}, $$
$$ \delta(x_{n}) = \sup_{\Vert x^{*} \Vert \leq 1} \inf_{m \in \mathbb{N}} \mathrm{diam} \{ x^{*}(x_{n}) : n \geq m \}. $$
Let $ C \geq 1 $. Following the authors of \cite{kalspur}, we say that a Banach space $ X $ has the \emph{$ C $-Schur property} if
$$ \mathrm{ca}(x_{n}) \leq C \delta(x_{n}) $$
for any bounded sequence $ x_{1}, x_{2}, \dots $ in $ X $.

\begin{proposition} \label{propETschur}
Let $ \T \in \mathrm{Tr}(2 \times \mathbb{N}) $.

{\rm (1)} If $ \forall \sigma \in 2^{\mathbb{N}} : \T(\sigma) \in \mathrm{IF} $, then $ E_{\T}^{*} $ has the $ 6 $-Schur property. Thus, $ E_{\T}^{*} $ has the Schur property and the Dunford-Pettis property.

{\rm (2)} In the opposite case, $ E_{\T}^{*} $ contains a complemented infinite-dimensional reflexive subspace. Thus, $ E_{\T}^{*} $ does not have the Schur property nor the Dunford-Pettis property.
\end{proposition}

\begin{proof}
The part (2) follows immediately from Lemma~\ref{lemmTssMT}(2) and Fact~\ref{factbranches}. Let us prove (1). Suppose that $ \forall \sigma \in 2^{\mathbb{N}} : \T(\sigma) \in \mathrm{IF} $ and that $ \mathrm{ca}(a^{*}_{n}) > 0 $ for a bounded sequence $ a^{*}_{1}, a^{*}_{2}, \dots $ in $ E_{\T}^{*} $. Let $ Q $ denote the non-empty set of all $ w^{*} $-cluster points of $ a^{*}_{n} $. To show that $ \mathrm{ca}(a^{*}_{n}) \leq 6 \delta(a^{*}_{n}) $, we consider two possibilities.

Assume first that $ \mathrm{diam} \, Q \geq \frac{1}{6} \mathrm{ca}(a^{*}_{n}) $. Let $ \varepsilon > 0 $ be given. There are $ a^{*}, b^{*} \in Q $ with $ \Vert b^{*} - a^{*} \Vert_{\T} > \frac{1}{6} \mathrm{ca}(a^{*}_{n}) - \varepsilon $. Further, there is $ x \in E_{\T}, \Vert x \Vert_{\T} \leq 1, $ such that $ (b^{*} - a^{*})(x) > \frac{1}{6} \mathrm{ca}(a^{*}_{n}) - \varepsilon $. For every $ m \in \mathbb{N} $, we can find $ k, l \geq m $ such that $ a^{*}_{k}(x) < a^{*}(x) + \varepsilon $ and $ a^{*}_{l}(x) > b^{*}(x) - \varepsilon $. Then $ \mathrm{diam} \{ a^{*}_{n}(x) : n \geq m \} \geq a^{*}_{l}(x) - a^{*}_{k}(x) > b^{*}(x) - a^{*}(x) - 2\varepsilon > \frac{1}{6} \mathrm{ca}(a^{*}_{n}) - 3\varepsilon $. It follows that $ \delta(a^{*}_{n}) \geq \frac{1}{6} \mathrm{ca}(a^{*}_{n}) - 3\varepsilon $. Since the argument works for any $ \varepsilon > 0 $, we obtain $ \delta(a^{*}_{n}) \geq \frac{1}{6} \mathrm{ca}(a^{*}_{n}) $.

Now, assume that $ \mathrm{diam} \, Q < \frac{1}{6} \mathrm{ca}(a^{*}_{n}) $. Notice that $ \mathrm{dist}(a^{*}_{n}, Q) \geq \frac{5}{12} \mathrm{ca}(a^{*}_{n}) $ for infinitely many $ n $'s. Therefore, we can find a subsequence $ a^{*}_{n_{k}} $ such that $ \mathrm{dist}(a^{*}_{n_{k}}, Q) \geq \frac{5}{12} \mathrm{ca}(a^{*}_{n}) $ for every $ k $. We may suppose that $ a^{*}_{n_{k}} $ converges to some $ a^{*} \in Q $ in the $ w^{*} $-topology. Let us put
$$ x^{*}_{k} = \frac{1}{\Vert a^{*}_{n_{k}} - a^{*} \Vert_{\T}} (a^{*}_{n_{k}} - a^{*}), \quad k = 1, 2, \dots. $$
By Lemma~\ref{lemmal1subseq}, there is a subsequence $ x^{*}_{k_{l}} $ such that
$$ \Big\Vert \sum_{l=1}^{m} \lambda_{l} x^{*}_{k_{l}} \Big\Vert_{\T} \geq \frac{1}{5} \sum_{l=1}^{m} |\lambda_{l}| $$
for all $ m \in \mathbb{N} $ and $ \lambda_{1}, \dots, \lambda_{m} \in \mathbb{R} $. Using the Hahn-Banach extension theorem, we can find $ x^{**} \in E_{\T}^{**} $ with $ \Vert x^{**} \Vert_{\T} \leq 1 $ such that
$$ x^{**}(x^{*}_{k_{l}}) = \frac{(-1)^{l}}{5}, \quad l = 1, 2, \dots. $$
Then
\begin{align*}
(-1)^{l} x^{**}(a^{*}_{n_{k_{l}}} - a^{*}) & = (-1)^{l} \Vert a^{*}_{n_{k_{l}}} - a^{*} \Vert_{\T} \, x^{**}(x^{*}_{k_{l}}) \\
 & = \frac{1}{5} \Vert a^{*}_{n_{k_{l}}} - a^{*} \Vert_{\T} \geq \frac{1}{5} \cdot \frac{5}{12} \mathrm{ca}(a^{*}_{n}) = \frac{1}{12} \mathrm{ca}(a^{*}_{n}),
\end{align*}
and so
$$ x^{**}(a^{*}_{n_{k_{2j}}}) \geq x^{**}(a^{*}) + \frac{1}{12} \mathrm{ca}(a^{*}_{n}), \quad x^{**}(a^{*}_{n_{k_{2j+1}}}) \leq x^{**}(a^{*}) - \frac{1}{12} \mathrm{ca}(a^{*}_{n}). $$
It follows that $ \delta(a^{*}_{n}) \geq \frac{1}{6} \mathrm{ca}(a^{*}_{n}) $.
\end{proof}

\begin{lemma} \label{lemmselectET}
There exist Borel mappings $ \sel, \sel^{*} : \mathrm{Tr}(2 \times \mathbb{N}) \to \mathcal{SE}(C([0, 1])) $ such that $ \sel(\T) $ is isometric to $ E_{\T} $ and $ \sel^{*}(\T) $ is isometric to $ E_{\T}^{*} $ for every $ \T \in \mathrm{Tr}(2 \times \mathbb{N}) $.
\end{lemma}

\begin{proof}
Let us prove the existence of $ \sel $ first. Due to Lemma~\ref{lemmselect}, it is sufficient to show that the function $ \T \mapsto \Vert x \Vert_{\T} $ is Borel for every $ x \in c_{00}(2^{< \mathbb{N}} \setminus \{ \emptyset \}) $. Let $ \Sigma $ be a finite subset of $ 2^{\mathbb{N}} $ such that $ x $ is supported by initial segments of elements of $ \Sigma $. By Fact~\ref{factMTcontin}, we have
$$ \Vert x \Vert_{\T} = \max_{\sigma \in \Sigma} \Big\Vert \sum_{\ell = 1}^{\infty} x(\sigma|_{\ell}) e_{\ell} \Big\Vert_{\mathcal{M}_{\T(\sigma)}}, \quad \T \in \mathrm{Tr}(2 \times \mathbb{N}). $$
For this reason, $ \T \mapsto \Vert x \Vert_{\T} $ is the maximum of finitely many continuous functions. Indeed,
\begin{itemize}
\item the mapping $ \T \mapsto \T(\sigma) $ is continuous from $ \mathrm{Tr}(2 \times \mathbb{N})$ into $ \mathrm{Tr}(\mathbb{N}) $ for every $ \sigma \in \Sigma $ (easy),
\item the mapping $ T \mapsto \mathcal{M}_{T} $ is continuous from $ \mathrm{Tr}(\mathbb{N}) $ into $ \mathcal{K}(2^{\mathbb{N}}) $ (just apply the fact that the Vietoris topology on $ \mathcal{K}(2^{\mathbb{N}}) $ is generated by the sets $ \{ \mathcal{M} \in \mathcal{K}(2^{\mathbb{N}}) : \mathcal{M} \cap \Delta_{\eta} \neq \emptyset \} $ and their complements, where $ \eta $ varies over sequences from $ 2^{< \mathbb{N}} $ and $ \Delta_{\eta} $ denotes the clopen set $ \{ \sigma \in 2^{\mathbb{N}} : \sigma|_{|\eta|} = \eta \} $),
\item the function $ \mathcal{M} \mapsto \Vert y \Vert_\mathcal{M} $ is continuous for every $ y \in c_{00} $ (see the proof of Lemma~\ref{lemmselectTss}).
\end{itemize}

Now, let us prove the existence of $ \sel^{*} $. By Lemma~\ref{lemmashrink}, the system $ \{ z_{\eta} : \eta \in 2^{< \mathbb{N}} \setminus \{ \emptyset \} \} $ is a shrinking basis of $ E_{\T} $ for every $ \T $. Using Lemma~\ref{lemmselect} again, it is therefore sufficient to show that the function $ \T \mapsto \Vert x^{*} \Vert_{\T} $ is Borel for every linear form $ x^{*} $ on $ c_{00}(2^{< \mathbb{N}} \setminus \{ \emptyset \}) $ with a finite support. Let $ S \subset c_{00}(2^{< \mathbb{N}} \setminus \{ \emptyset \}) $ be the countable set of all non-zero vectors with rational coordinates. Then
$$ \Vert x^{*} \Vert_{\T} = \sup_{x \in S} \frac{x^{*}(x)}{\Vert x \Vert_{\T}}, \quad \T \in \mathrm{Tr}(2 \times \mathbb{N}). $$
It follows from the first part of the proof that $ \T \mapsto \Vert x^{*} \Vert_{\T} $ is Borel.
\end{proof}

\begin{proof}[Proof of Theorem \ref{thmschur}]
The class of all separable Banach spaces with the Schur property is $ \mathbf{\Pi}^{1}_{2} $ (\cite[Theorem~28]{braga}), as well as the class of all separable Banach spaces with the Dunford-Pettis property (Lemma~\ref{lemmDP}). Let us show that both classes are $ \mathbf{\Pi}^{1}_{2} $-hard. Let $ \sel^{*} $ be a mapping provided by Lemma~\ref{lemmselectET}. Using Proposition~\ref{propETschur}, we obtain
\begin{align*}
\forall \sigma \in 2^{\mathbb{N}} : \T(\sigma) \in \mathrm{IF} \quad & \Leftrightarrow \quad \sel^{*}(\T) \textrm{ has the Schur property} \\
& \Leftrightarrow \quad \sel^{*}(\T) \textrm{ has the Dunford-Pettis property}
\end{align*}
for every $ \T \in \mathrm{Tr}(2 \times \mathbb{N}) $. It means that $ \sel^{*}(\T) $ has the Schur property (Dunford-Pettis property) if and only if $ \T \in C $, where $ C \subset \mathrm{Tr}(2 \times \mathbb{N}) $ is the $ \mathbf{\Pi}^{1}_{2} $-complete set from Lemma~\ref{lemmCPCA}. It remains to apply Lemma~\ref{lemmhardset}.
\end{proof}

\begin{remark}
(i) It is possible to use \cite[Theorem~3.2]{behrends} to quantify the Schur property of $ E_{\T}^{*} $ in a bit different direction. If the assumption $ \forall \sigma \in 2^{\mathbb{N}} : \T(\sigma) \in \mathrm{IF} $ is met, then $ E_{\T}^{*} $ has the strong Schur property in the sense that every bounded sequence $ a^{*}_{1}, a^{*}_{2}, \dots $ in $ E_{\T}^{*} $ with $ \inf \{ \Vert a^{*}_{n} - a^{*}_{m} \Vert_{\T} : n \neq m \} > \varepsilon $ has a subsequence $ a^{*}_{n_{j}} $ such that
$$ \Big\Vert \sum_{j=1}^{k} \lambda_{j} a^{*}_{n_{j}} \Big\Vert_{\T} \geq \frac{\varepsilon}{12} \sum_{j=1}^{k} |\lambda_{j}| $$
for all $ k \in \mathbb{N} $ and $ \lambda_{1}, \dots, \lambda_{k} \in \mathbb{R} $.

(ii) If the assumption $ \forall \sigma \in 2^{\mathbb{N}} : \T(\sigma) \in \mathrm{IF} $ is met, a quantitative version of the Dunford-Pettis property of $ E_{\T}^{*} $ and of $ E_{\T} $ can be obtained as well (see \cite[Proposition~6.4 and Theorem~6.5]{kackalspur}).

(iii) In \cite{kurka}, a question was considered whether the proposed tree space method can be used for amalgamating of spaces with the Schur property (see \cite[Remark~3.7(c)]{kurka} for the exact formulation). Proposition~\ref{propETschur}(1) shows a concrete example of a family of spaces with the Schur property for which the answer is positive, although not trivial. Let us note that to prove that $ E_{\T}^{*} $ has the Schur property would be much simpler if we had the positive answer to the following question: \emph{Does a Banach space $ X $ has necessarily the Schur property if it has a subset $ W $ such that $ \overline{\mathrm{co}} \, W = B_{X} $ and every weakly convergent sequence of elements of $ W $ is convergent in the norm?}
\end{remark}

\section{Proof of Lemma \ref{lemmal1subseq}}

Let $ \T \in \mathrm{Tr}(2 \times \mathbb{N}) $ satisfying  $ \forall \sigma \in 2^{\mathbb{N}} : \T(\sigma) \in \mathrm{IF} $ be given, together with a normalized sequence $ x^{*}_{1}, x^{*}_{2}, \dots $ in $ E_{\T}^{*} $ converging to $ 0 $ in the $ w^{*} $-topology. Let us recall that our task is to find a subsequence $ y^{*}_{1}, y^{*}_{2}, \dots $ such that
$$ \Big\Vert \sum_{k=1}^{n} \lambda_{k} y^{*}_{k} \Big\Vert_{\T} \geq \frac{1}{5} \sum_{k=1}^{n} |\lambda_{k}| $$
for all $ n \in \mathbb{N} $ and $ \lambda_{1}, \dots, \lambda_{n} \in \mathbb{R} $.

Note that each $ x^{*} \in E_{\T}^{*} $ can be viewed as the system $ \{ x^{*}(z_{\eta}) \}_{\eta \in 2^{< \mathbb{N}} \setminus \{ \emptyset \}} $ of real numbers. By Lemma~\ref{lemmashrink}, elements with a finite support are dense in $ E_{\T}^{*} $. Note also that $ x^{*}_{k}(z_{\eta}) \to 0 $ for every $ \eta $.

By the passage to a subsequence and a small perturbation, we can obtain a sequence (which is denoted also $ x^{*}_{k} $) satisfying:

(WLOG-1) There are $ 1 \leq p_{1} \leq q_{1} < p_{2} \leq q_{2} < \dots $ such that $ x^{*}_{k} $ is supported by sequences of length in $ [p_{k}, q_{k}] $. (Because of the perturbation, we just need to prove the desired inequality with a constant better than $ \frac{1}{5} $).

For every $ k $, let $ x_{k} \in E_{\T} $ be such that $ x^{*}_{k}(x_{k}) = \Vert x_{k} \Vert_{\T} = 1 $ and $ x_{k} $ is supported by sequences of length in $ [p_{k}, q_{k}] $ (as well as $ x^{*}_{k} $).
Let
$$ x_{k} = \sum_{\eta \in 2^{p_{k}}} x_{k, \eta} $$
be the decomposition of $ x_{k} $ such that $ x_{k, \eta} $ is supported by sequences which extend $ \eta $.

Let us denote $ \Delta = 2^{\mathbb{N}} $ and $ \Delta_{\eta} = \{ \sigma \in \Delta : \sigma|_{|\eta|} = \eta \} $. Let $ \Sigma_{\ell} $ be the $ \sigma $-algebra generated by the sets $ \Delta_{\eta}, \eta \in 2^{\ell} $. The formula
$$ \m_{k}(\Delta_{\eta}) = x^{*}_{k}(x_{k, \eta}), \quad \eta \in 2^{p_{k}}, $$
defines a probability measure on $ \Sigma_{p_{k}} $. (We have $ x^{*}_{k}(x_{k, \eta}) \geq 0 $ because $ 1 - x^{*}_{k}(x_{k, \eta}) = x^{*}_{k}(x_{k} - x_{k, \eta}) \leq \Vert x_{k} - x_{k, \eta} \Vert_{\T} \leq \Vert x_{k} \Vert_{\T} = 1 $).

Every $ \m_{k} $ can be extended to a Borel probability measure on $ \Delta $. The sequence of these extensions has a cluster point in the $ w^{*} $-topology of $ C(\Delta)^{*} $. We can therefore assume that:

(WLOG-2) The measures $ \m_{k} $ converge to a Borel probability measure $ \m $ on $ \Delta $ in the sense that $ \m_{k}(\Delta_{\eta}) \to \m(\Delta_{\eta}) $ for every $ \eta \in 2^{< \mathbb{N}} $.

\begin{claim} \label{claiml1subseq1}
There is an increasing sequence $ s_{1} < s_{2} < \dots $ of natural numbers and a closed subset $ \Gamma \subset \Delta $ such that $ \m(\Gamma) \geq \frac{7}{8} $ and, for every $ \sigma \in \Gamma $, the system $ \mathcal{M}_{\T(\sigma)} $ contains a set which intersects $ [s_{n}, s_{n+1}) $ for each $ n \in \mathbb{N} $.
\end{claim}

\begin{proof}
The set $ [\T] $ of all infinite branches of $ \T $ is a closed subset of $ 2^{\mathbb{N}} \times \mathbb{N}^{\mathbb{N}} $ whose sections $ [\T(\sigma)], \sigma \in 2^{\mathbb{N}}, $ are non-empty due to the assumption of the lemma. By the Jankov-von Neumann uniformization theorem (see e.g. (18.1) in \cite{kechris}), there exists a selector $ \sigma \in 2^{\mathbb{N}} \mapsto \nu_{\sigma} \in [\T(\sigma)] $ which is measurable with respect to the $ \sigma $-algebra generated by the analytic subsets of $ 2^{\mathbb{N}} $. By a theorem of Lusin (see e.g. (21.10) in \cite{kechris}), members of this $ \sigma $-algebra are $ \overline{\m} $-measurable, where $ \overline{\m} $ denotes the completion of $ \m $.

For natural numbers $ r \leq s $, let us denote
$$ \Lambda_{r, s} = \{ \sigma \in \Delta : [r, s) \cap \tilde{\nu}_{\sigma} \neq \emptyset \}. $$
For every $ r \in \mathbb{N} $ and $ \varepsilon > 0 $, since $ \bigcup_{s=r}^{\infty} \Lambda_{r, s} = \Delta $, there is $ s \geq r $ such that $ \overline{\m}(\Lambda_{r, s}) \geq 1 - \varepsilon $.

Let us take a sequence $ \varepsilon_{1}, \varepsilon_{2}, \dots $ of positive numbers such that $ \sum_{n=1}^{\infty} \varepsilon_{n} < \frac{1}{8} $. Let $ s_{1} = 1 $ and let $ s_{2}, s_{3}, \dots $ be chosen in the way that
$$ \overline{\m}(\Lambda_{s_{n}, s_{n+1}}) \geq 1 - \varepsilon_{n} $$
for $ n = 1, 2, \dots $. The set
$$ \Gamma_{0} = \bigcap_{n=1}^{\infty} \Lambda_{s_{n}, s_{n+1}} $$
fulfills $ \overline{\m}(\Gamma_{0}) > 1 - \frac{1}{8} = \frac{7}{8} $ and, for every $ \sigma \in \Gamma_{0} $, the system $ \mathcal{M}_{\T(\sigma)} $ contains the set $ \tilde{\nu}_{\sigma} $ which intersects $ [s_{n}, s_{n+1}) $ for each $ n \in \mathbb{N} $. Finally, let $ \Gamma \subset \Gamma_{0} $ be a compact subset with $ \overline{\m}(\Gamma) \geq \frac{7}{8} $.
\end{proof}

Now, let us consider such $ s_{1} < s_{2} < \dots $ and $ \Gamma \subset \Delta $ as in Claim~\ref{claiml1subseq1}. Let $ \theta \subset 2^{< \mathbb{N}} \setminus \{ \emptyset \} $ denote the set of all non-empty initial segments of sequences from $ \Gamma $. Let
$$ x_{k} = u_{k} + v_{k} $$
be the decomposition of $ x_{k} $ such that $ u_{k} $ is supported by $ \theta $ and $ v_{k} $ is supported by the complement of $ \theta $. Let
$$ x_{k, \eta} = u_{k, \eta} + v_{k, \eta} $$
be the analogous decomposition of $ x_{k, \eta} $.

We are ready to establish our third and last additional assumption.

(WLOG-3) One of the following possibilities takes place:
\begin{itemize}
\item[{(I)}] $ x^{*}_{k}(u_{k}) \geq \frac{1}{2} $ for every $ k $.
\item[{(II)}] $ x^{*}_{k}(v_{k}) > \frac{1}{2} $ for every $ k $.
\end{itemize}

\begin{claim} \label{claiml1subseq2}
There is a subsequence $ y^{*}_{j} $ of $ x^{*}_{k} $ such that, for every $ m \in \mathbb{N} $, there is $ w \in E_{\T} $ with $ \Vert w \Vert_{\T} \leq 1 $ and
$$ y^{*}_{j}(w) \geq \frac{1}{4}, \quad j = 1, 2, \dots, m. $$
\end{claim}

Before the proof of this claim, we show that the provided subsequence $ y^{*}_{j} $ of $ x^{*}_{k} $ has the desired property. Given $ m \in \mathbb{N} $ and $ \lambda_{1}, \dots, \lambda_{m} \in \mathbb{R} $, taking a suitable $ w $ and using that $ x^{*}_{1}, x^{*}_{2}, \dots $ have disjoint supports, we obtain
$$ \Big\Vert \sum_{j=1}^{m} \lambda_{j} y^{*}_{j} \Big\Vert_{\T} = \Big\Vert \sum_{j=1}^{m} |\lambda_{j}| y^{*}_{j} \Big\Vert_{\T} \geq \Big\Vert \sum_{j=1}^{m} |\lambda_{j}| y^{*}_{j} \Big\Vert_{\T} \Vert w \Vert_{\T} $$
$$ \geq \Big( \sum_{j=1}^{m} |\lambda_{j}| y^{*}_{j} \Big) (w) = \sum_{j=1}^{m} |\lambda_{j}| y^{*}_{j}(w) \geq \frac{1}{4} \sum_{j=1}^{m} |\lambda_{j}|. $$
Let us recall that a better constant than $ \frac{1}{5} $ is needed because of the perturbation done at the beginning of this section. As the constant $ \frac{1}{4} $ is greater than $ \frac{1}{5} $, Lemma~\ref{lemmal1subseq} is proven.

It remains to prove Claim ~\ref{claiml1subseq2}. We consider separately the possibilities (I) and (II) introduced above.

\begin{proof}[Proof of Claim~\ref{claiml1subseq2}, case (I)]
We choose a subsequence $ y^{*}_{j} = x^{*}_{k_{j}} $ in the way that
$$ s_{n_{1}} < s_{n_{1}+1} < p_{k_{1}} \leq q_{k_{1}} < s_{n_{2}} < s_{n_{2}+1} < p_{k_{2}} \leq q_{k_{2}} < s_{n_{3}} < s_{n_{3}+1} < \dots $$
for some suitable $ n_{1}, n_{2}, \dots $. Let us consider the intervals in $ \mathbb{N} $ given by
$$ I_{j} = [p_{k_{j}}, q_{k_{j}}], \quad j = 1, 2, \dots. $$
Due to the choice of $ s_{1} < s_{2} < \dots $ and $ \Gamma \subset \Delta $ (see Claim~\ref{claiml1subseq1}), the family $ \{ I_{1}, \dots, I_{m} \} $ is $ \mathcal{M}_{\T(\sigma)} $-admissible for every $ m \in \mathbb{N} $ and every $ \sigma \in \Gamma $.

Given $ m \in \mathbb{N} $, let us define
$$ w = \frac{1}{2} \sum_{i=1}^{m} u_{k_{i}}. $$
Using (I), we obtain for $ 1 \leq j \leq m $ that
$$ y^{*}_{j}(w) = x^{*}_{k_{j}}(w) = \frac{1}{2} \sum_{i=1}^{m} x^{*}_{k_{j}}(u_{k_{i}}) = \frac{1}{2} \, x^{*}_{k_{j}}(u_{k_{j}}) \geq \frac{1}{2} \cdot \frac{1}{2} = \frac{1}{4}, $$
so it is sufficient to verify that $ \Vert w \Vert_{\T} \leq 1 $, i.e., that
$$ \forall \sigma \in 2^{\mathbb{N}} : \quad \Big\Vert \sum_{\ell = 1}^{\infty} w(\sigma|_{\ell}) e_{\ell} \Big\Vert_{\mathcal{M}_{\T(\sigma)}} \leq 1. $$

Consider $ \sigma \in \Gamma $ first. For $ 1 \leq i \leq m $, the point $ \sum_{\ell = 1}^{\infty} u_{k_{i}}(\sigma|_{\ell}) e_{\ell} $ is supported by $ [p_{k_{i}}, q_{k_{i}}] = I_{i} $. Using $ \mathcal{M}_{\T(\sigma)} $-admissibility of $ \{ I_{1}, \dots, I_{m} \} $ and Fact~\ref{factsupineq}, we obtain
\begin{align*}
\Big\Vert \sum_{\ell = 1}^{\infty} w(\sigma|_{\ell}) e_{\ell} \Big\Vert_{\mathcal{M}_{\T(\sigma)}} & = \frac{1}{2} \Big\Vert \sum_{i=1}^{m} \sum_{\ell = 1}^{\infty} u_{k_{i}}(\sigma|_{\ell}) e_{\ell} \Big\Vert_{\mathcal{M}_{\T(\sigma)}} \\
 & \leq \sup_{1 \leq i \leq m} \Big\Vert \sum_{\ell = 1}^{\infty} u_{k_{i}}(\sigma|_{\ell}) e_{\ell} \Big\Vert_{\mathcal{M}_{\T(\sigma)}} \\
 & \leq \sup_{1 \leq i \leq m} \Vert u_{k_{i}} \Vert_{\T} \leq \sup_{1 \leq i \leq m} \Vert x_{k_{i}} \Vert_{\T} = 1.
\end{align*}

Now, consider $ \sigma \in \Delta \setminus \Gamma $. Note that $ w $ is supported by $ \theta $ and that $ \sigma \notin \Gamma = \overline{\Gamma} = [\theta \cup \{ \emptyset \}] $. Let $ \eta $ be the longest initial segment of $ \sigma $ belonging to $ \theta \cup \{ \emptyset \} $. Then $ \eta $ is an initial segment of some $ \sigma' \in \Gamma $. Due to Fact~\ref{factMTcontin}, we obtain
\begin{align*}
\Big\Vert \sum_{\ell = 1}^{\infty} & w(\sigma|_{\ell}) e_{\ell} \Big\Vert_{\mathcal{M}_{\T(\sigma)}} = \Big\Vert \sum_{\ell = 1}^{|\eta|} w(\sigma|_{\ell}) e_{\ell} \Big\Vert_{\mathcal{M}_{\T(\sigma)}} \\
 & = \Big\Vert \sum_{\ell = 1}^{|\eta|} w(\sigma'|_{\ell}) e_{\ell} \Big\Vert_{\mathcal{M}_{\T(\sigma')}} \leq \Big\Vert \sum_{\ell = 1}^{\infty} w(\sigma'|_{\ell}) e_{\ell} \Big\Vert_{\mathcal{M}_{\T(\sigma')}} \leq 1.
\end{align*}
This completes the verification of $ \Vert w \Vert_{\T} \leq 1 $.
\end{proof}

\begin{proof}[Proof of Claim~\ref{claiml1subseq2}, case (II)]
Let $ \Gamma^{(\ell)} $ denote the smallest set in $ \Sigma_{\ell} $ containing $ \Gamma $, that is $ \Gamma^{(\ell)} = \{ \sigma \in \Delta : \sigma|_{\ell} \in \theta \} $. We choose a subsequence $ y^{*}_{j} = x^{*}_{k_{j}} $ in the way that
$$ \big| \m_{k_{j+1}}(\Gamma^{(q_{k_{j}})}) - \m(\Gamma^{(q_{k_{j}})}) \big| \leq \frac{1}{8}, \quad j = 1, 2, \dots. $$
We have
$$ \m_{k_{j+1}}(\Gamma^{(q_{k_{j}})}) \geq \m(\Gamma^{(q_{k_{j}})}) - \frac{1}{8} \geq \m(\Gamma) - \frac{1}{8} \geq \frac{7}{8} - \frac{1}{8} = \frac{3}{4}. $$
Let us define $ w_{k_{1}} = v_{k_{1}} $ and
$$ w_{k_{j+1}} = v_{k_{j+1}} - \sum_{\eta \in 2^{p_{k_{j+1}}}, \Delta_{\eta} \cap \Gamma^{(q_{k_{j}})} = \emptyset} v_{k_{j+1}, \eta}, \quad j = 1, 2, \dots. $$
Then, using (II), we obtain $ x^{*}_{k_{1}}(w_{k_{1}}) = x^{*}_{k_{1}}(v_{k_{1}}) \geq \frac{1}{2} \geq \frac{1}{4} $ and
\begin{align*}
x^{*}_{k_{j+1}}(w_{k_{j+1}}) & = x^{*}_{k_{j+1}}(v_{k_{j+1}}) - \sum_{\eta \in 2^{p_{k_{j+1}}}, \Delta_{\eta} \cap \Gamma^{(q_{k_{j}})} = \emptyset} x^{*}_{k_{j+1}}(v_{k_{j+1}, \eta}) \\
 & \geq \frac{1}{2} - \sum_{\eta \in 2^{p_{k_{j+1}}}, \Delta_{\eta} \cap \Gamma^{(q_{k_{j}})} = \emptyset} x^{*}_{k_{j+1}}(x_{k_{j+1}, \eta}) \\
 & = \frac{1}{2} - \sum_{\eta \in 2^{p_{k_{j+1}}}, \Delta_{\eta} \cap \Gamma^{(q_{k_{j}})} = \emptyset} \m_{k_{j+1}}(\Delta_{\eta}) \\
 & = \frac{1}{2} - \m_{k_{j+1}}(\Delta \setminus \Gamma^{(q_{k_{j}})}) \geq \frac{1}{2} - \frac{1}{4} = \frac{1}{4}.
\end{align*}
(We have $ x^{*}_{k_{j+1}}(v_{k_{j+1}, \eta}) \leq x^{*}_{k_{j+1}}(x_{k_{j+1}, \eta}) $ because $ 1 - x^{*}_{k_{j+1}}(u_{k_{j+1}, \eta}) = x^{*}_{k_{j+1}}(x_{k_{j+1}} - u_{k_{j+1}, \eta}) \leq \Vert x_{k_{j+1}} - u_{k_{j+1}, \eta} \Vert_{\T} \leq \Vert x_{k_{j+1}} \Vert_{\T} = 1 $).

We claim that every infinite branch intersects the support of at most one $ w_{k_{j}} $. Let us make two observations first.

(a) If $ \sigma \in \Gamma^{(q_{k_{i}})} $, then the branch $ \{ \sigma|_{1}, \sigma|_{2}, \dots \} $ does not intersect the support of $ w_{k_{j}} $ for $ j \leq i $.

Indeed, the initial segments $ \sigma|_{1}, \sigma|_{2}, \dots, \sigma|_{q_{k_{i}}} $ belong to $ \theta $. In particular, the initial segments $ \sigma|_{p_{k_{j}}}, \sigma|_{p_{k_{j}}+1}, \dots, \sigma|_{q_{k_{j}}} $ belong to $ \theta $. The support of $ w_{k_{j}} $ is disjoint from $ \theta $, and so $ \{ \sigma|_{1}, \sigma|_{2}, \dots \} $ does not intersect the support of $ w_{k_{j}} $.

(b) If $ \sigma \notin \Gamma^{(q_{k_{i}})} $, then the branch $ \{ \sigma|_{1}, \sigma|_{2}, \dots \} $ does not intersect the support of $ w_{k_{j+1}} $ for $ j \geq i $.

Indeed, we have
$$ \Delta_{\sigma|_{q_{k_{i}}}} \cap \Gamma^{(q_{k_{i}})} = \emptyset $$
and, in particular,
$$ \Delta_{\sigma|_{p_{k_{j+1}}}} \cap \Gamma^{(q_{k_{j}})} = \emptyset. $$
The sequence $ \eta = \sigma|_{p_{k_{j+1}}} $ appears in the sum in the definition of $ w_{k_{j+1}} $, and so $ w_{k_{j+1}}(\sigma|_{\ell}) = 0 $ for every $ \ell \geq p_{k_{j+1}} $.

Now, we obtain from (a) and (b) that
\begin{itemize}
\item if $ \sigma \in \Delta \setminus \Gamma^{(q_{k_{1}})} $, then the branch $ \{ \sigma|_{1}, \sigma|_{2}, \dots \} $ does not intersect the support of $ w_{k_{j}} $ for $ j \neq 1 $,
\item if $ \sigma \in \Gamma^{(q_{k_{i}})} \setminus \Gamma^{(q_{k_{i+1}})} $ for some $ i $, then the branch $ \{ \sigma|_{1}, \sigma|_{2}, \dots \} $ does not intersect the support of $ w_{k_{j}} $ for $ j \neq i+1 $,
\item if $ \sigma \in \bigcap_{i=1}^{\infty} \Gamma^{(q_{k_{i}})} $, then the branch $ \{ \sigma|_{1}, \sigma|_{2}, \dots \} $ does not intersect the support of $ w_{k_{j}} $ for every $ j $.
\end{itemize}

So, we have shown that every infinite branch intersects the support of at most one $ w_{k_{j}} $. Now, given $ m \in \mathbb{N} $, let us define
$$ w = \sum_{i=1}^{m} w_{k_{i}}. $$
For every $ \sigma \in 2^{\mathbb{N}} $, there is $ j \leq m $ such that $ w(\sigma|_{\ell}) = w_{k_{j}}(\sigma|_{\ell}) $ for each $ \ell \in \mathbb{N} $ (we can choose any $ j \leq m $ if $ w(\sigma|_{\ell}) = 0 $ for each $ \ell $), and so
$$ \Big\Vert \sum_{\ell = 1}^{\infty} w(\sigma|_{\ell}) e_{\ell} \Big\Vert_{\mathcal{M}_{\T(\sigma)}} = \Big\Vert \sum_{\ell = 1}^{\infty} w_{k_{j}} (\sigma|_{\ell}) e_{\ell} \Big\Vert_{\mathcal{M}_{\T(\sigma)}} \leq \Vert w_{k_{j}} \Vert_{\T} \leq \Vert x_{k_{j}} \Vert_{\T} = 1. $$
It follows that $ \Vert w \Vert_{\T} \leq 1 $. At the same time, for $ j \leq m $, we have
$$ y^{*}_{j}(w) = x^{*}_{k_{j}}(w) = \sum_{i=1}^{m} x^{*}_{k_{j}}(w_{k_{i}}) = x^{*}_{k_{j}}(w_{k_{j}}) \geq \frac{1}{4}, $$
and thus $ w $ works.
\end{proof}

\section{A Question} \label{sec:question}

The aim of this short final section is a discussion on the complexity of the isomorphism class of $ c_{0} $ (see Question~\ref{questgodef}) and the formulation of a related problem concerning equivalent norms on $ c_{0} $. First, let us mention a remarkable conjecture from \cite{gokala}.

\begin{conjecture}[Godefroy, Kalton, Lancien]
If $ X $ is a Banach space with summable Szlenk index whose dual $ X^{*} $ is isomorphic to $ \ell_{1} $, then $ X $ is isomorphic to $ c_{0} $.
\end{conjecture}

The validity of this conjecture would imply that every Banach space uniformly isomorphic to $ c_{0} $ was actually isomorphic to $ c_{0} $. As noted by G.~Godefroy \cite{godefroyprobl}, confirming the conjecture would give also the positive answer to Question~\ref{questgodef}. We are going to provide a variant of this approach.

We start with an investigation of subspaces of $ c_{0} $ based on a result of N.~J.~Kalton \cite{kalton}. Let us recall that we denote
$$ \mathrm{ca}(x_{n}) = \inf_{m \in \mathbb{N}} \mathrm{diam} \{ x_{n} : n \geq m \} $$
for a bounded sequence $ x_{1}, x_{2}, \dots $ in a Banach space $ X $.

\begin{theorem}
For a separable Banach space $ X $, the following assertions are equivalent:
\begin{itemize}
\item[{\rm (i)}] $ X $ can be embedded isomorphically into $ c_{0} $.
\item[{\rm (ii)}] There exists a bounded function $ \mu : B_{X^{*}} \to \mathbb{R} $ such that
$$ \liminf_{n \to \infty} \mu(x^{*} + x^{*}_{n}) \geq \mu(x^{*}) + \liminf_{n \to \infty} \Vert x^{*}_{n} \Vert $$
whenever $ x^{*} \in B_{X^{*}} $, $ x^{*} + x^{*}_{n} \in B_{X^{*}} $ and $ x^{*}_{1}, x^{*}_{2}, \dots $ is $ w^{*} $-null.
\item[{\rm (iii)}] There exists a bounded function $ \pi : B_{X^{*}} \to \mathbb{R} $ such that
$$ \pi(x^{*}) \geq \mathrm{ca}(x^{*}_{n}) + \liminf_{n \to \infty} \pi(x^{*} + x^{*}_{n}) $$
whenever $ x^{*} \in B_{X^{*}} $, $ x^{*} + x^{*}_{n} \in B_{X^{*}} $ and $ x^{*}_{1}, x^{*}_{2}, \dots $ is $ w^{*} $-null.
\item[{\rm (iv)}] We have $ \pi_{X}^{\alpha} < \infty $ on $ B_{X^{*}} $ for every $ \alpha < \omega_{1} $ where $ \pi_{X}^{0}(x^{*}) = 0 $,
\begin{align*}
\pi_{X}^{\alpha + 1}(x^{*}) = \sup \big\{ \mathrm{ca}(x^{*}_{n}) & + \liminf_{n \to \infty} \pi_{X}^{\alpha}(x^{*} + x^{*}_{n}) : \\
 & x^{*} + x^{*}_{n} \in B_{X^{*}} \textrm{ and } x^{*}_{n} \xrightarrow{w^{*}} 0 \big\}
\end{align*}
and
$$ \pi_{X}^{\beta}(x^{*}) = \lim_{\alpha \nearrow \beta} \pi_{X}^{\alpha}(x^{*}) \quad \textrm{if $ \beta $ is limit.} $$
\end{itemize}
\end{theorem}

Note that it is possible to replace \textquotedblleft liminf\textquotedblright {} with \textquotedblleft limsup\textquotedblright {} in the condition (ii). Of course, there is a version of (iv) based on $ \mu $ instead of $ \pi $, we prefer the current version nevertheless, as the sequence $ \pi_{X}^{\alpha} $ is related to the Szlenk derivatives studied in \cite{bossard}. It can be shown that a separable Banach space $ X $ has summable Szlenk index if and only if $ \pi_{X}^{\omega} < \infty $ on $ B_{X^{*}} $.

\begin{proof}[Proof (sketch).]
(i) $ \Rightarrow $ (ii): It is known that the function $ \mu(x^{*}) = \Vert x^{*} \Vert $ has the desired property if $ X $ is isometric to a subspace of $ c_{0} $ \cite{kalwer}.

(ii) $ \Rightarrow $ (i): This is a consequence of \cite[Theorem~3.3]{kalton}.

(ii) $ \Rightarrow $ (iii): The function $ \pi = -2\mu $ works.

(iii) $ \Rightarrow $ (ii): The function $ \mu = -\pi $ works.

(iii) $ \Rightarrow $ (iv): We can assume that $ \pi \geq 0 $, in which case $ \pi_{X}^{\alpha} \leq \pi $.

(iv) $ \Rightarrow $ (iii): The function $ \pi(x^{*}) = \lim_{\alpha \nearrow \omega_{1}} \pi_{X}^{\alpha}(x^{*}) $ works.
\end{proof}

Now, we proceed to our problem.

\begin{question}
Does there exist some $ \alpha < \omega_{1} $ such that $ \pi_{X}^{\alpha} = \pi_{X}^{\alpha + 1} $ for every Banach space $ X $ isomorphic to $ c_{0} $?
\end{question}

By a classical result of W.~B.~Johnson and M.~Zippin \cite{johzip}, a subspace of $ c_{0} $ is isomorphic to $ c_{0} $ if and only if it is an $ \mathcal{L}_{\infty} $-space. Therefore, if such $ \alpha < \omega_{1} $ as in the question exists, then a separable Banach space $ X $ is isomorphic to $ c_{0} $ if and only if it is an $ \mathcal{L}_{\infty} $-space, $ \pi_{X}^{\alpha} < \infty $ and $ \pi_{X}^{\alpha} = \pi_{X}^{\alpha + 1} $.

Considering results from Section~\ref{sec:tsirelson}, we expect that the answer is negative for the class of all spaces which are isomorphic to the $ c_{0} $-sum of a sequence of finite-dimensional spaces and, in particular, for the class of all spaces which can be embedded isomorphically into $ c_{0} $.

\end{document}